\documentclass{amsart}
\pdfoutput=1

\usepackage[backend=biber,maxbibnames=5,maxalphanames=5,style=alphabetic,bibencoding=utf8,giveninits,url=false,isbn=false]{biblatex}
\AtBeginBibliography{\small}

\usepackage{lineno}
\usepackage[all,hyperref]{paper_diening}

\usepackage{pgfplots}
\usepackage{tikz}
\usepackage{pgf}

\newcommand{\nodes}{\mathcal{N}}

\newcommand{\tria}{\mathcal{T}}

\newcommand{\dx}{\,\mathrm{d}x}
\newcommand{\ds}{\,\mathrm{d}s}

\makeatletter
\@namedef{subjclassname@2020}{%
  \textup{2020} Mathematics Subject Classification}
\makeatother

\newcommand*\linenomathpatch[1]{%
  \cspreto{#1}{\linenomath}%
  \cspreto{#1*}{\linenomath}%
  \csappto{end#1}{\endlinenomath}%
  \csappto{end#1*}{\endlinenomath}%
}

\linenomathpatch{equation}
\linenomathpatch{gather}
\linenomathpatch{multline}
\linenomathpatch{align}
\linenomathpatch{alignat}
\linenomathpatch{flalign}

\begin{document}

\author[R.\ Stevenson]{Rob Stevenson}
\address[R.\ Stevenson]{Korteweg-de Vries (KdV) Institute for Mathematics, University of Amsterdam, PO Box 94248, 1090 GE Amsterdam, The Netherlands}
\email{r.p.stevenson@uva.nl}

\author[J.\ Storn]{Johannes Storn}
\address[J.\ Storn]{Department of Mathematics, Bielefeld University, Postfach 10 01 31, 33501 Bielefeld, Germany}
\email{jstorn@math.uni-bielefeld.de}

\thanks{
This research has been supported by the Netherlands Organization for Scientific Research (NWO) under contract.\ no.\ SH-208-11, by the NSF Grant DMS ID 1720297, and by the Deutsche Forschungsgemeinschaft (DFG, German Research Foundation) -- SFB 1283/2 2021 -- 317210226.}


\title{Interpolation operators for parabolic Problems}

\begin{abstract}
We introduce interpolation operators with approximation and stability properties suited for parabolic problems in primal and mixed formulations. We derive localized error estimates for tensor product meshes (occurring in classical time-marching schemes) as well as locally in space-time refined meshes.
\end{abstract}

\maketitle

\section{Introduction}
In recent years simultaneous space-time variational formulations for parabolic problems became more and more popular. Besides practical aspects like highly parallelizable computations \cite{DyjaGanapathysubramanianVanderZee18,NeumuellerSmears19,HoferLangerNeumuellerSchneckenleitner19,VenetieWesterdiep21}
the ansatz offers analytical advantages including quasi-optimality of the discrete solution \cite{TantardiniVeeser16} (also called symmetric error estimates in \cite{DupontLiu02,ChrysafinosWalkington06}). This property motivates adaptive time stepping \cite{Feischl22}, adaptive wavelet schemes \cite{RekatsinasStevenson19}, adaptive wavelet-in-time and finite-element-in-space approaches \cite{StevensonVenetieWesterdiep21}, and even adaptive mesh refinements locally in space-time \cite{LangerSchafelner20,LangerSteinbachTroltzschYang21,DieningStorn21,GantnerStevenson22}. While numerical experiments suggest superiority of the latter approach for singular solutions, theoretical results are restricted to plain convergence \cite{GantnerStevenson20} but do not verify optimal convergence rates as they do for elliptic problems \cite{Stevenson07,CarstensenFeischlPagePraetorius14}.
Motivated by the extension of such optimality results to parabolic problem, this paper introduces and investigates a main ingredient in the analysis of adaptive schemes for parabolic problems like the heat equation in a time-space cylinder $Q = \mathcal{J} \times \Omega$, namely interpolation operators suited for the norm 
\begin{align*}
\lVert \bigcdot \rVert_X \coloneqq  
\big(\lVert \partial_t \, \bigcdot \rVert_{L^2(\mathcal{J};H^{-1}(\Omega))}^2  + \lVert \nabla_x \bigcdot \rVert_{L^2(Q)}^2\big)^{1/2}.
\end{align*} 
Additionally, we introduce an interpolation operator for first-order formulations of the heat equation satisfying a beneficial commuting diagram property.
On tensor product meshes the interpolation operators are stable and have optimal approximation properties. 
We give upper bounds for the interpolation errors and emphasize the need of parabolic scaling if the solution is rough in time. The localization of the interpolation error in space leads to unavoidable weights in terms of negative powers of the local mesh size. Under realistic regularity assumptions we can overcome these negative powers due to parabolic scaling. Unfortunately, this strategy cannot be applied to the interpolation error of adaptively refined meshes. In fact, we illustrate that any (local) interpolation operator experiences these difficulties.
Overall, this paper's main contributions are the following. 
\begin{itemize}
\item We present approximation properties suited for parabolic problems in Section~\ref{sec:LocEstimates}--\ref{sec:interpolInTimeAndSpace}.
\item We introduce an interpolation operator with optimal approximation properties on tensor product meshes in Section~\ref{sec:IX}. 
\item We introduce an interpolation operator suited for first-order formulations with optimal approximation properties on tensor product meshes and a commuting diagram property in Section~\ref{subsec:ILambda}.
\item We introduce an interpolation operator for locally in space-time refined meshes and discuss its stability in Section~\ref{sec:IXirreg}.
\end{itemize}

\section{Bochner spaces and their discretization}\label{sec:Discreatization}
This section introduces Bochner spaces, suitable discretizations by finite elements, and their underlying partitions.
\subsection{Bochner spaces}
Our analysis is motivated by the approximation of parabolic problems like the heat equation. Given a time-space cylinder with bounded time interval $\mathcal{J} = [0,T] \subset \mathbb{R}^d$ and bounded Lipschitz domain $\Omega \subset \mathbb{R}^d$, this problem seeks with given right-hand side $f\colon Q\to \mathbb{R}$ and initial data $u_0 \colon \Omega \to \mathbb{R}$ the solution $u \colon Q \to \mathbb{R}$ to
\begin{align}\label{eq:Heat}
\partial_t u - \Delta_x u = f\text{ in }Q,\qquad u = 0 \text{ on }\mathcal{J} \times \partial \Omega,\qquad u(0) = u_0\text{ in }\Omega.
\end{align}
A suitable analytical setting relies on Sobolev-Bochner spaces. 
Therefore, we set the space $H^{-1}(\Omega)$ as the dual of the Sobolev space $H^1_0(\Omega)$ equipped with norm $\lVert \nabla_x \bigcdot \rVert_{L^2(\Omega)}$ and dual pairing $\langle \bigcdot ,\bigcdot\rangle_{\Omega} \coloneqq \langle \bigcdot ,\bigcdot\rangle_{H^{-1},H^1_0(\Omega)}$ which equals the $L^2$ inner product for smooth functions.
Given $V \in \lbrace H^1_0(\Omega), L^2(\Omega);H^{-1}(\Omega) \rbrace$, we set
\begin{align*}
\begin{aligned}
\lVert p \rVert_{L^2(\mathcal{J};V)}^2 & \coloneqq \int_\mathcal{J} \lVert p(s)\rVert_V^2\ds&&\text{for all }p\colon \mathcal{J} \to V,\\
\lVert v \rVert_{H^1(\mathcal{J};V)}^2 & \coloneqq\lVert v \rVert_{L^2(\mathcal{J};V)}^2 + \lVert \partial_t v \rVert_{L^2(\mathcal{J};V)}^2&&\text{for all }v\colon \mathcal{J} \to V.\\
\end{aligned}
\end{align*}
The Bochner spaces read
\begin{align*}
\begin{aligned}
L^2(\mathcal{J};V) & \coloneqq \lbrace p\colon \mathcal{J} \to V\colon \lVert p \rVert_{L^2(\mathcal{J};V)} < \infty \rbrace,\\
H^1(\mathcal{J};V) & \coloneqq \lbrace v \colon \mathcal{J} \to V \colon  \lVert v \rVert_{H^1(\mathcal{J};V)}  < \infty \rbrace.
\end{aligned}
\end{align*} 
We can identify $L^2(\mathcal{J};L^2(\Omega)) = L^2(Q)$. Moreover, we have the following.
\begin{remark}[Tensor spaces]\label{rem:TensorBochner}
Bochner spaces can be seen as closure of algebraic tensor product spaces \cite[Rem.~64.24]{ErnGuermond21c}, i.e., for $V \in \lbrace H^1_0(\Omega), L^2(\Omega);H^{-1}(\Omega)\rbrace$
\begin{align*}
L^2(\mathcal{J}) \otimes V &\coloneqq \textup{span}\lbrace v_tv_x \colon v_t\in L^2(\mathcal{J})\text{ and }v_x \in V\rbrace&& \text{ is dense in }L^2(\mathcal{J};V),\\
H^1(\mathcal{J}) \otimes V &\coloneqq \textup{span}\lbrace v_tv_x \colon v_t\in H^1(\mathcal{J})\text{ and }v_x \in V\rbrace&& \text{ is dense in }H^1(\mathcal{J};V).
\end{align*}
\end{remark}
We are particularly interested in the space
\begin{align}\label{eq:defX}
X \coloneqq L^2(\mathcal{J};H^1_0(\Omega)) \cap H^1(\mathcal{J};H^{-1}(\Omega)).
\end{align} 
\begin{lemma}[Embedding]\label{lem:embedding}
We have for all $v\in X$ and $t\in \mathcal{J} = [0,T]$ 
\begin{align*}
\lVert v(t) \rVert_{L^2(\Omega)}^2 \leq T^{-1} \lVert v\rVert_{L^2(Q)}^2 + \lVert\nabla_x v\rVert_{L^2(Q)}^2  +  \lVert \partial_t v \rVert_{L^2(\mathcal{J};H^{-1}(\Omega))}^2.
\end{align*}
\end{lemma}
\begin{proof}
This is a known result which we prove to stress the dependency on $T$ often hidden in textbooks. Let $v\in X$ and $t\in \mathcal{J}$. The fundamental theorem of calculus \cite[Thm.~64.31]{ErnGuermond21c} reveals for all $\tau \in \mathcal{J}$
\begin{align*}
\lVert v(t) \rVert_{L^2(\Omega)}^2& = \lVert v(\tau) \rVert_{L^2(\Omega)}^2 + 2 \int_t^\tau \langle \partial_t v,v\rangle_\Omega \ds \\
&\leq \lVert v(\tau) \rVert_{L^2(\Omega)}^2 + \int_t^\tau \big( \lVert \partial_t v(s) \rVert_{H^{-1}(\Omega)}^2 + \lVert \nabla_x v(s) \rVert_{L^2(\Omega)}^2 \big) \ds\\
& \leq \lVert v(\tau) \rVert_{L^2(\Omega)}^2 + \lVert \partial_t v  \rVert_{L^2(\mathcal{J};H^{-1}(\Omega))}^2 + \lVert \nabla_x v\rVert_{L^2(Q)}^2.
\end{align*}
An integration of the inequality over all $\tau \in \mathcal{J}$ concludes the proof.
\end{proof}
Given a right-hand side $f\in L^2(\mathcal{J};H^{-1}(\Omega))$ and initial data $u_0\in L^2(\Omega)$, the problem in \eqref{eq:Heat} has a unique solution $u\in X$ \cite[Thm.~5.1]{SchwabStevenson09}.
More precisely, the mapping $(f,u_0) \mapsto u$ is a linear isomorphism and so the norm of $u$ depends continuously on the data, that is,
\begin{align}\label{eq:Lis}
\begin{aligned}
\lVert u \rVert_X^2 & := \lVert u \rVert_{L^2(\mathcal{J};H^1_0(\Omega))}^2 + \lVert \partial_t u \rVert_{L^2(\mathcal{J};H^{-1}(\Omega))}^2\\
&\hphantom{:} = \lVert \nabla_x u \rVert_{L^2(Q)}^2 + \lVert \partial_t u \rVert_{L^2(\mathcal{J};H^{-1}(\Omega))}^2 \eqsim \lVert f \rVert_{L^2(\mathcal{J};H^{-1}(\Omega))}^2 + \lVert u_0 \rVert^2_{L^2(\Omega)}.
\end{aligned}
\end{align}
If the right-hand side is slightly smoother in space, that is  $f \in L^2(\mathcal{J};L^2(\Omega))$, we have for initial data $u_0 \in H^1_0(\Omega)$ the additional regularity property \cite[Sec.\ 4]{Dier15}
\begin{align}\label{eq:Reg1}
\lVert \Delta_x u \rVert_{L^2(Q)}^2 + \lVert \partial_t u \rVert_{L^2(Q)}^2 \lesssim \lVert f \rVert^2_{L^2(Q)} + \lVert \nabla_x u_0 \rVert_{L^2(\Omega)}^2.
\end{align}
If $f \in H^1(\mathcal{J};H^{-1}(\Omega))$ and $f(0) + \Delta_x u_0 \in L^2(\Omega)$, then $\xi = \partial_t u$ solves
\begin{align*}
\partial_t \xi - \Delta_x \xi = \partial_t f\text{ in }Q,\qquad \xi = 0 \text{ on }\mathcal{J} \times \partial \Omega,\qquad \xi(0) = f(0) + \Delta_x u_0\text{ in }\Omega. 
\end{align*}
Thus, \eqref{eq:Lis} leads to the bound
\begin{align}\label{eq:Reg2}
\begin{aligned}
&\lVert  \partial_t \nabla_x u \rVert_{L^2(Q)}^2 + \lVert \partial_t^2 u \rVert_{L^2(\mathcal{J};H^{-1}(\Omega))}^2\\
&\qquad\qquad\qquad\qquad \lesssim  \lVert \partial_t f \rVert_{L^2(\mathcal{J};H^{-1}(\Omega))}^2 + \lVert f(0) + \Delta_x u_0\rVert_{L^2(\Omega)}^2.
\end{aligned}
\end{align} 
Notice that  elliptic regularity results imply for convex or smooth domains $\Omega$ 
\begin{align}\label{eq:Reg3}
\lVert \nabla^2_x u \rVert_{L^2(Q)} \lesssim \lVert \Delta_x u \rVert_{L^2(Q)}.
\end{align}
The estimates in \eqref{eq:Reg1}--\eqref{eq:Reg3} provide some reasonable regularity assumptions.
\subsection{Triangulation}\label{subsec:tria}
Rather than using simplicial partitions of the time-space cylinder $Q = \mathcal{J} \times \Omega \subset \mathbb{R}^{d+1}$, we use partitions $\tria$ of $Q$ into cylindrical closed time-space cells $K=K_t\times K_x$ with time interval $K_t\subset \mathbb{R}$ and simplices $K_x\subset \mathbb{R}^{d}$ as in \cite{DieningStorn21,GantnerStevenson22}.
The following considerations motivate the use of such partitions.
\begin{itemize}
\item A special case of cylindrical partitions are tensor product meshes which typically occur in time-marching schemes and are thus of great interest.
\item The parabolic \Poincare inequality in Theorem~\ref{thm:paraPoincare} suggests the use of parabolically scaled meshes for irregular solutions. Thus, we want to allow for local mesh refinements such that the diameter of local cells in space direction $h_x$ and the length of cells in time direction $h_t$ satisfy
\begin{align*}
\begin{aligned}
h_t &\eqsim h_x^2&&\text{if we scale parabolically},\\
h_t &\eqsim h_x&&\text{if we scale equally}.
\end{aligned}
\end{align*}
Such refinements can easily be achieved with cylindrical meshes.
\item  The faces of each time-space cell in a cylindrical partition $\tria$ are either parallel or perpendicular to the time axis. This allows for the design of finite elements that are better suited for approximations in spaces like $L^2(\mathcal{J};H(\textup{div}_x,\Omega)) = \lbrace \tau \in L^2(Q;\mathbb{R}^d)\colon \textup{div}_x\, \tau \in L^2(Q)\rbrace$, where $\textup{div}_x$ denotes the divergence in space. This leads to significantly improved rates of convergence compared to finite elements on simplicial meshes; see \cite{GantnerStevenson22}.
\end{itemize} 
Throughout this paper we suppose that the partition $\tria$ of $Q = \mathcal{J} \times \Omega$ consists of time-space cells $K = K_t \times K_x \subset \mathbb{R}^{d+1}$ with shape regular $d$-simplices $K_x$. 
A special class of meshes satisfying these assumptions are tensor-product meshes. Given conforming partitions $\tria_t$ and $\tria_x$ of the time interval $\mathcal{J}$ and the domain $\Omega$ into shape-regular simplices, these meshes read
\begin{align}\label{eq:defTensorMesh}
\tria_\otimes = \tria_t \otimes \tria_x = \lbrace K_t \times K_x\colon K_t\in \tria_t\text{ and }K_x\in \tria_x\rbrace.
\end{align}
Besides these tensor product meshes, we discuss adaptively refined meshes with hanging vertices in Section~\ref{sec:IXirreg}.
\subsection{Finite element spaces}
Let $\tria$ be a partition of $Q$ as described in the previous subsection.
For all cells $K=K_t \times K_x \in \tria$ and polynomial degrees $k\in \mathbb{N}_0$ we set for $L \in \lbrace K,K_t,K_x\rbrace$ the space of polynomials
\begin{align*}
\mathbb{P}_k(L)\coloneqq \lbrace v_h \in L^2(L)\colon v_h\text{ is a polynomial of maximal degree }k\rbrace.
\end{align*}
Given polynomial degrees $k,\ell \in \mathbb{N}$, we discretize the space $X$ in \eqref{eq:defX} by
\begin{align}\label{eq:DefXh}
X_h \coloneqq X_h^{k,\ell} \coloneqq \lbrace v_h \in X\colon v_h|_K \in \mathbb{P}_k(K_t)\otimes \mathbb{P}_\ell(K_x)\text{ for all }K_t\times K_x \in \tria\rbrace.
\end{align}
A special class of meshes included in our analysis are tensor product meshes $\tria_\otimes = \tria_t \otimes \tria_x$ introduced in \eqref{eq:defTensorMesh}. We set the spaces
\begin{align}\label{eq:defFEspaces}
\begin{aligned}
\mathcal{L}^0_k(\tria_t) &\coloneqq \lbrace p_t \in L^2(\mathcal{J})\colon p_t|_{K_t} \in \mathbb{P}_k(K_t)\text{ for all }K_t \in \tria_t\rbrace,\\
\mathcal{L}^1_k(\tria_t) &\coloneqq \lbrace v_t \in H^1(\mathcal{J})\colon v_t|_{K_t} \in \mathbb{P}_k(K_t)\text{ for all }K_t \in \tria_t\rbrace,\\
\mathcal{L}^1_{\ell}(\tria_x) &\coloneqq \lbrace v_x \in H^1(\Omega)\colon v_x|_{K_x} \in \mathbb{P}_\ell(K_x)\text{ for all }K_x \in \tria_x\rbrace,\\
\mathcal{L}^1_{\ell,0}(\tria_x) &\coloneqq \lbrace v_x \in H^1_0(\Omega)\colon v_x|_{K_x} \in \mathbb{P}_\ell(K_x)\text{ for all }K_x \in \tria_x\rbrace.
\end{aligned}
\end{align}
If $\tria = \tria_\otimes$ is a tensor product mesh, the ansatz space in \eqref{eq:DefXh} equals
\begin{align*}
X_h = X_h^{k,\ell} = \mathcal{L}^1_k(\tria_t;\mathcal{L}^1_{\ell,0}(\tria_x)) \coloneqq \mathcal{L}^1_k(\tria_t)\otimes \mathcal{L}^1_{\ell,0}(\tria_x).
\end{align*}
\section{Local estimates}\label{sec:LocEstimates}
In this section we introduce several local estimates for functions on a time-space cell $K = K_t \times K_x$. The cell consists of a bounded time-interval $K_t \subset \mathbb{R}$ of length $h_t \coloneqq |K_t| >0$ and a simplex $K_x \subset \mathbb{R}^d$ with diameter $h_x \coloneqq \textup{diam}(K_x)$.
The space $H^{-1}(K_x)$ is defined as the dual of $H^1_0(K_x)$ with dual pairing $\langle \bigcdot ,\bigcdot\rangle_{K_x} \coloneqq \langle \bigcdot ,\bigcdot\rangle_{H^{-1}(K_x);H^1_0(K_x)}$ and dual norm
\begin{align*}
\lVert \xi \rVert_{H^{-1}(K_x)} \coloneqq \sup_{w\in H^1_0(K_x)} \frac{\langle \xi,w\rangle_{K_x}}{\lVert \nabla_x w\rVert_{L^2(K_x)}}\qquad\text{for all }\xi \in H^{-1}(K_x).
\end{align*}
This definition and Friedrichs' inequality lead to the upper bound
\begin{align}\label{eq:FriedrichsConsequance}
\lVert f \rVert_{H^{-1}(K_x)} \lesssim h_x \lVert f \rVert_{L^2(K_x)}\qquad\text{for all }f \in L^2(K_x).
\end{align}
The following lemma shows that these two terms are equivalent for polynomials.
\begin{lemma}[Inverse estimate]\label{lem:inverseEst}
Let $k\in \mathbb{N}_0$.
We have the upper bound
\begin{align*}
\lVert f_h \rVert_{L^2(K_x)} \lesssim h_x^{-1} \lVert f_h\rVert_{H^{-1}(K_x)}\qquad\text{for all }f_h \in \mathbb{P}_k(K_x). 
\end{align*}
The hidden constant depends solely on the degree $k$ and the shape regularity of $K_x$.
\end{lemma}
\begin{proof}
The proof can be found in \cite[Lem.~1]{FuehrerHeuerKarkulik21}.
\end{proof}
The following result is of crucial importance for the analysis of parabolic problems. It involves the integral mean
\begin{align*}
\langle f \rangle_{K} \coloneqq \dashint_K f \dx\qquad\text{for all }f\in L^2(K).
\end{align*}
The result is stated in a very general formulation in \cite[Lem.~2.9]{DieningSchwarzacherStroffoliniVerde17}. 
Rather than using the more general result, we give an alternative direct proof.
\begin{theorem}[Parabolic \Poincare{} inequality]\label{thm:paraPoincare}
All functions $v\in L^2(K_t;H^1(K_x)) \cap H^1(K_t;H^{-1}(K_x))$ satisfy
\begin{align*}
\lVert v - \langle v \rangle_K \rVert_{L^2(K)} \lesssim h_x\, \lVert \nabla_x v \rVert_{L^2(K)} + \frac{h_t}{h_x}\, \lVert \partial_t v\rVert_{L^2(K_t;H^{-1}(K_x))}.
\end{align*}
More general, we have for $k,\ell\in \mathbb{N}_0$ 
\begin{align*}
\min_{v_h \in \mathbb{P}_k(K_t;\mathbb{P}_\ell(K_x))}\lVert v - v_h \rVert_{L^2(K)} &\lesssim h_x\, \min_{v_x \in L^2(K_t;\mathbb{P}_\ell(K_x))} \lVert \nabla_x (v-v_x) \rVert_{L^2(K)}\\
&\quad + \frac{h_t}{h_x}\, \min_{v_t\in \mathbb{P}_k(K_t;H^{-1}(K_x))} \lVert \partial_t (v-v_t)\rVert_{L^2(K_t;H^{-1}(K_x))}. 
\end{align*}
The hidden constant depends solely on the polynomial degrees $k$ and $\ell$ as well as the shape regularity of $K_x$.
\end{theorem}
The proof of the theorem splits the approximation of $v$ by a polynomial on $K$ into the approximation by a polynomial in time and a polynomial in space.
While approximation properties of the latter are well understood, we state approximation properties of functions in $\mathbb{P}_k(K_t;H^{-1}(K_x)) = \mathbb{P}_k(K_t) \otimes H^{-1}(K_x)$.
\begin{lemma}[Averaged Taylor polynomial in time]\label{lem:avgTaylor}
Let $\xi \in H^k(K_t;V)$ with $V \in \lbrace L^2(K_x),H^{-1}(K_x)\rbrace$ and $k \in \mathbb{N}_0$. There exists a polynomial $\xi_h \in \mathbb{P}_k(K_t;V)$ with
\begin{align*}
\lVert \partial_t^m (\xi - \xi_h)\rVert_{L^2(K_t;V)} \lesssim h_t^{k-m}\, \lVert  \partial_t^k \xi \rVert_{L^2(K_t;V)}\qquad\text{for all }m=0,\dots,k.
\end{align*}
\end{lemma}
\begin{proof}
This result follows directly from the tensor product structure in Remark~\ref{rem:TensorBochner} and approximation properties of polynomials in $H^m(K_t)$. A detailed proof (for general $L^p$ spaces with $p \in [1,\infty]$) can be found in the appendix of \cite{DieningStornTscherpel22}.
\end{proof}
Let $\mathcal{I}^{L^2}_t \colon L^2(K_t) \to \mathbb{P}_k(K_t)$ be an $L^2(K_t)$ stable projection onto the space of polynomials of maximal degree $k\in \mathbb{N}_0$ in time. Its application everywhere in space  leads to a mapping for functions on the entire time-space cell $K$, that is,
\begin{align*}
\mathcal{I}^{L^2}_t\colon L^2(K_t;H^{-1}(K_x)) \to \mathbb{P}_k(K_t;H^{-1}(K_x)).
\end{align*}
\begin{lemma}[Approximability in $L^2(K_t;H^{-1}(K_x))$]\label{lem:ApxInL2Hmin}
The mapping $\mathcal{I}^{L^2}_t\colon L^2(K) \to \mathbb{P}_k(K_t;H^{-1}(K_x))$  satisfies for all $v \in H^m(K_t;H^{-1}(K_x))$ and $m=0,\dots,k$
\begin{align*}
\lVert \partial_t^m (v - \mathcal{I}^{L^2}_t v) \rVert_{L^2(K_t;H^{-1}(K_x))} \eqsim \min_{v_t \in \mathbb{P}_k(K_t;H^{-1}(K_x))} \lVert \partial_t^m (v - v_t) \rVert_{L^2(K_t;H^{-1}(K_x))}.
\end{align*}
\end{lemma}
\begin{proof}
This result follows by classical arguments using Lemma~\ref{lem:avgTaylor}. See \cite[Thm.~24]{DieningStornTscherpel22} for a detailed proof.
\end{proof}
With these two results we are able to verify Theorem~\ref{thm:paraPoincare}.
\begin{proof}[Proof of Theorem~\ref{thm:paraPoincare}]
We denote the $L^2(K_t)$ orthogonal projection in time and the $H^{-1}(K_x)$ orthogonal projection in space onto constant functions by
\begin{align*}
\Pi_{L^2(K_t)}\colon L^2(K_t) \to \mathbb{P}_0(K_t)\qquad\text{and}\qquad \Pi_{H^{-1}(K_x)}\colon H^{-1}(K_x)\to \mathbb{P}_0(K_x). 
\end{align*}
By applying them everywhere in time or space they extend to semi-discrete maps
\begin{align*}
\Pi_{L^2(K_t)}&\colon L^2(K_t;L^2(K_x)) \to \mathbb{P}_0(K_t;L^2(K_x)),\\
\Pi_{H^{-1}(K_x)}&\colon L^2(K_t;H^{-1}(K_x))\to L^2(K_t;\mathbb{P}_0(K_x)). 
\end{align*}
Since their composition maps onto constant functions, we have for all $f\in L^2(K)$
\begin{align}\label{eq:ProofTemp1}
\begin{aligned}
&\lVert f - \langle f \rangle_K\rVert_{L^2(K)}  \leq \lVert f - \Pi_{H^{-1}(K_x)}\Pi_{L^2(K_t)} f \rVert_{L^2(K)}\\
&\qquad \leq  \lVert f - \Pi_{H^{-1}(K_x)} f \rVert_{L^2(K)}  + \lVert  \Pi_{H^{-1}(K_x)}  (f - \Pi_{L^2(K_t)} f )\rVert_{L^2(K)}.
\end{aligned}
\end{align}
Set $\langle f \rangle_{K_x} \coloneqq \dashint_{K_x}f\dx \in L^2(K_t;\mathbb{P}_0(K_x))$. The first addend is bounded by
\begin{align*}
\lVert f - \Pi_{H^{-1}(K_x)} f \rVert_{L^2(K)} \leq \lVert f - \langle f\rangle_{K_x} \rVert_{L^2(K)} + \lVert \Pi_{H^{-1}(K_x)} (f - \langle f\rangle_{K_x}) \rVert_{L^2(K)}.
\end{align*}
The inverse estimate in Lemma~\ref{lem:inverseEst} yields $L^2$ stability of $\Pi_{H^{-1}(K_x)}$ in the sense that
\begin{align*}
\lVert \Pi_{H^{-1}(K_x)} g \rVert_{L^2(K)} &\lesssim h_x^{-1} \lVert \Pi_{H^{-1}(K_x)} g \rVert_{L^2(K_t;H^{-1}(K_x))} \\
& \leq  h_x^{-1} \lVert g \rVert_{L^2(K_t;H^{-1}(K_x))} \lesssim \lVert  g \rVert_{L^2(K)}\qquad\text{for all }g\in L^2(K).
\end{align*}
These two estimates (with $g = f - \langle f\rangle_{K_x}$) and \Poincare's inequality show
\begin{align}\label{eq:Profasadasd}
\lVert f - \Pi_{H^{-1}(K_x)} f \rVert_{L^2(K)} \lesssim h_x\, \lVert \nabla_x f \rVert_{L^2(K)}.
\end{align}
The second addend in \eqref{eq:ProofTemp1} is bounded due to the inverse estimate in Lemma~\ref{lem:inverseEst}, stability of $\Pi_{H^{-1}(K_x)}$ in $H^{-1}(K_x)$, and approximation properties in Lemma~\ref{lem:avgTaylor}--\ref{lem:ApxInL2Hmin} by
\begin{align}\label{eq:Profasadasd2}
\begin{aligned}
 \lVert  \Pi_{H^{-1}(K_x)}  (f - \Pi_{L^2(K_t)} f )\rVert_{L^2(K)}& \lesssim  h_x^{-1} \lVert f - \Pi_{L^2(K_t)} f \rVert_{L^2(K_t;H^{-1}(K_x))} \\
 &\lesssim h_t h^{-1}_x \lVert \partial_t f\rVert_{L^2(K_t;H^{-1}(K_x))}. 
\end{aligned}
\end{align}
Combining \eqref{eq:ProofTemp1}--\eqref{eq:Profasadasd2} concludes the proof of the first inequality in the theorem. Similar arguments yield the second inequality.
\end{proof}
If the function $v$ in Theorem~\ref{thm:paraPoincare} satisfies additionally that $\partial_t v \in L^2(K)$, an application of \eqref{eq:FriedrichsConsequance} to the first estimate leads to the \Poincare inequality 
\begin{align}\label{eq:PoincareClassic}
\lVert v - \langle v \rangle_K\rVert_{L^2(K)} \lesssim h_x\, \lVert \nabla_x v \rVert_{L^2(K)} + h_t\, \lVert \partial_t v \rVert_{L^2(K)}. 
\end{align}
In this regard Theorem~\ref{thm:paraPoincare} can be seen as a weaker version of \Poincare's inequality that is better suited for parabolic problems.
For example, the regularity stated in \eqref{eq:Reg1} does not yield $\partial_t \nabla_x u\in L^2(K)$ for the solution to the heat equation, preventing an application of \eqref{eq:PoincareClassic}. However, Theorem~\ref{thm:paraPoincare} applies and yields with parabolic scaling $h_t \eqsim h_x^2$ the convergence result
\begin{align*}
\lVert \nabla_x u - \langle \nabla_x u\rangle_K\rVert_{L^2(K)}&\lesssim h_x\, \lVert \nabla_x^2 u \rVert_{L^2(K)} + \frac{h_t}{h_x}\, \lVert \partial_t \nabla_x u \rVert_{L^2(K_t;H^{-1}(K_x))} \\
&\lesssim h_x\, (\lVert \nabla_x^2 u \rVert_{L^2(K)} +  \lVert \partial_t u \rVert_{L^2(K)}).  
\end{align*}
The need of parabolic scaling for irregular solutions is further illustrated by the numerical experiment in \cite[Sec.~7.4]{DieningStorn21}.
\begin{remark}[Sharp estimate]
Inverse estimates show that the bound in Theorem~\ref{thm:paraPoincare} must be sharp. More precisely, let $v = v_t v_x$ with polynomials $v_t\in \mathbb{P}_k(K_t)$ and $v_x\in \mathbb{P}_\ell(K_x)$ with $\langle v\rangle_K = 0$ for $K=K_t\times K_x$. Then inverse estimates reveal
\begin{align*}
h_x\, \lVert \nabla_x v \rVert_{L^2(K)} + \frac{h_t}{h_x} \lVert \partial_t v \rVert_{L^2(K_t;H^{-1}(K_x))} \lesssim \lVert v\rVert_{L^2(K)} + h_t\, \lVert \partial_t v\rVert_{L^2(K)} \lesssim \lVert v \rVert_{L^2(K)}.  
\end{align*}

\end{remark}
\section{Interpolation in space or time}\label{sec:interpolInTimeAndSpace}
The main idea in this paper's design of interpolation operators in space-time is to exploit the tensor product structure of Bochner spaces like $H^1(\mathcal{J};H^{-1}(\Omega)) = H^1(\mathcal{J}) \otimes H^{-1}(\Omega) \supset X$. This allows us to apply an interpolation operator in time to the $H^1(\mathcal{J})$ component and in space to the $H^{-1}(\Omega)$ component. 
\subsection{Interpolation operator in space}\label{sec:interHmin}
We utilize the $H^{-1}(\Omega)$ stable interpolation operator $\mathcal{I}_x\colon H^{-1}(\Omega) \to \mathcal{L}^1_{\ell,0}(\tria_x)$ introduced in \cite{DieningStornTscherpel22} for conforming and shape-regular partitions $\tria_x$ of $\Omega$ with $\ell \in \mathbb{N}$. Throughout this subsection we assume that $\tria_x$ is such a partition. 
Let $\nodes_x$ denote the set of vertices in $\tria_x$ and set for all $j\in \nodes_x$ the corresponding vertex patch
\begin{align*}
\omega_{x,j} \coloneqq \bigcup \lbrace K_x\in \tria_x\colon j\in K_x\rbrace.
\end{align*}
We denote the nodal basis functions by $\phi_{x,j} \in \mathcal{L}^1_1(\tria_x)$ with $\phi_{x,j}(i) = \delta_{i,j}$ for all vertices $i,j\in \nodes_x$. 
\begin{lemma}[Localization of $H^{-1}(\Omega)$]\label{lem:LocHminOne}
Let $\xi \in H^{-1}(\Omega)$. Then we have 
\begin{align*}
\sum_{j\in \nodes_x} \lVert \xi \rVert_{H^{-1}(\omega_{x,j})}^2 \lesssim \lVert \xi \rVert_{H^{-1}(\Omega)}^2 \lesssim \sum_{j\in \nodes_x} h_{x,j}^{-2} \lVert \xi \rVert_{H^{-1}(\omega_{x,j})}^2.
\end{align*}
\end{lemma}
\begin{proof}
Let $\xi \in H^{-1}(\Omega)$. The partition of unity $1 = \sum_{j\in \nodes_x} \phi_{x,j}$ leads for all $w\in H^1_0(\Omega)$ to the upper bound
\begin{align}\label{eq:proofTemp3}
\begin{aligned}
\langle \xi ,w\rangle_\Omega &= \sum_{j\in \nodes_x} \langle \xi, \phi_{x,j} w\rangle_\Omega\\
& \leq \sum_{j\in \nodes_x} h_{x,j}^{-1}\lVert \xi \rVert_{H^{-1}(\omega_{x,j})} h_{x,j}\lVert \nabla_x( \phi_{x,j} w)\rVert_{L^2(\omega_{x,j})} \\
& \lesssim \Big(\sum_{j\in \nodes_x} h_{x,j}^{-2} \lVert \xi \rVert_{H^{-1}(\omega_{x,j})}^2\Big)^{1/2} \bigg(\Big(\sum_{j\in \nodes_x} h_{x,j}^2 \lVert \nabla_x w \rVert_{L^2(\omega_{x,j})}^2\Big)^{1/2}\\
&\qquad\qquad\qquad\qquad\qquad\qquad\qquad\qquad + \Big(\sum_{j\in \nodes_x} \lVert  w \rVert_{L^2(\omega_{x,j})}^2\Big)^{1/2}\bigg) \\
&\lesssim \Big(\sum_{j\in \nodes_x} h_{x,j}^{-2} \lVert \xi \rVert_{H^{-1}(\omega_{x,j})}^2\Big)^{1/2} \lVert \nabla_x w \rVert_{L^2(\Omega)}.
\end{aligned}
\end{align}
This concludes the proof of the upper bound. The lower bound follows with standard arguments (see for example \cite[Lem.~11]{DieningStornTscherpel22}).
\end{proof}
The upper bound in Lemma~\ref{lem:LocHminOne} is indeed sharp, as one can see by localizing the $H^{-1}(\Omega)$ norm of the constant function $\xi = 1 \in H^{-1}(\Omega)$. 
The operator $\mathcal{I}_x$ allows for a localization of the $H^{-1}(\Omega)$ norm without any additional weights. In particular, we have the following result involving the patches
\begin{align*}
\begin{aligned}
\omega_{x,j}^2 &\coloneqq \bigcup \Big\lbrace \omega_{x,i}\colon i \in \omega_{x,j} \rbrace&&\text{for all }j\in \nodes_x,\\
\omega_{K_x} &\coloneqq \bigcup \Big\lbrace K_x'\in \tria_x \colon K_x \cap K_x'\neq \emptyset\rbrace&&\text{for all }K_x \in \tria_x.
\end{aligned}
\end{align*}
\begin{theorem}[Interpolation operator $\mathcal{I}_x$]\label{thm:Ix}
The operator $\mathcal{I}_x \colon H^{-1}(\Omega) \to \mathcal{L}^1_{\ell,0}(\tria_x)$ is a linear projection onto $\mathcal{L}^1_{\ell,0}(\tria_x)$. It satisfies for all $\xi \in H^{-1}(\Omega)$ 
\begin{align*}
\lVert \xi - \mathcal{I}_x \xi \rVert^2_{H^{-1}(\Omega)}& \eqsim  \sum_{j\in \nodes_x} \lVert \xi - \mathcal{I}_x \xi \rVert_{H^{-1}(\omega_{x,j})}^2\eqsim\sum_{j\in \nodes_x} \min_{\xi_h\in \mathcal{L}^1_{\ell,0}(\tria_x)} \lVert \xi - \xi_h\rVert_{H^{-1}(\omega_{x,j}^2)}^2.
\end{align*}
Moreover, it satisfies for all $\xi \in L^2(\Omega)$  and $K\in \tria$
\begin{align*}
\lVert \xi - \mathcal{I}_x \xi\rVert_{L^2(K)} \eqsim \min_{\xi_h \in \mathcal{L}^1_{\ell,0}(\tria_x)} \lVert \xi - \xi_h\rVert_{L^2(\omega_{K_x})}.
\end{align*}
\end{theorem}
\begin{proof}
This result is shown in \cite[Thm.~1]{DieningStornTscherpel22}.
\end{proof}
\begin{remark}[Boundary data]
It is possible to modify the design of $\mathcal{I}_x$ in order to replace the space $\mathcal{L}^1_{\ell,0}(\tria_x)$ equipped with zero boundary data by the space $\mathcal{L}^1_{\ell}(\tria_x)$ without zero boundary data; see \cite{DieningStornTscherpel22} for details. 
\end{remark}
An application of $\mathcal{I}_x$ everywhere in time extends the operator to a mapping $\mathcal{I}_x \colon L^2(\mathcal{J};H^{-1}(\Omega)) \to L^2(\mathcal{J};\mathcal{L}^1_{\ell,0}(\tria_x))$ in the sense that for all $v\in L^2(\mathcal{J};H^{-1}(\Omega))$ 
\begin{align}\label{eq:defIxJ}
(\mathcal{I}_x v)(s) = \mathcal{I}_x v(s)\qquad\text{for almost all }s \in \mathcal{J}.
\end{align}
\subsection{Interpolation operators in time}\label{sec:It}
Besides the interpolation operator $\mathcal{I}_x$ in space introduced in the previous subsection, we utilize an interpolation operator $\mathcal{I}_t\colon H^1(\mathcal{J}) \to \mathcal{L}^1_k(\tria_t)$ with polynomial degree $k\in \mathbb{N}$ and partition $\tria_t$ of the time interval $\mathcal{J}$. We set the operator locally for each for time interval $K_t = [a,b] \in \tria_t$. Its definition involves the bubble function $b_{K_t} \in \mathbb{P}_2(K_t)$ with $\int_{K_t} b_{K_t} \ds = 1$ and $b_{K_t}(a) = 0 = b_{K_t}(b)$.
For $v\in H^1(K_t)$ we set the operator as follows. Let $\mathcal{I}_{K_t}^1 v\in \mathbb{P}_1(K_t)$ denote the nodal interpolation defined by
\begin{align*}
\mathcal{I}_{K_t}^1 v(a)= v(a)\qquad\text{and}\qquad\mathcal{I}_{K_t}^1 v(b) = v(b).
\end{align*} 
Let $\mathcal{I}_{K_t}^2v = 0$ for $k=1$ and for $k\geq 2$ let $\mathcal{I}_{K_t}^2 v \in \mathbb{P}_{k-2}(K_t)$ be the solution to
\begin{align*}
\int_{K_t} b_{K_t} (\mathcal{I}_{K_t}^2 v)\,  w_{k-2} \ds = \int_{K_t} (v-\mathcal{I}_{K_t}^1 v)\,  w_{k-2} \ds \qquad\text{for all }w_{k-2}\in \mathbb{P}_{k-2}(K_t).
\end{align*}
We set the interpolation of $v$ as
\begin{align}\label{eq:DefIt}
\mathcal{I}_{K_t} v \coloneqq \mathcal{I}_{K_t}^1 v + b_{K_t} \mathcal{I}_{K_t}^2 v.
\end{align}
Moreover, we denote the $L^2(K_t)$ orthogonal projection onto $\mathbb{P}_{r}(K_t)$ by
\begin{align}\label{eq:DefPit}
\Pi_{\mathbb{P}_r(K_t)}\colon L^2(K_t) \to \mathbb{P}_r(K_t)\qquad\text{for all }r\in \mathbb{N}_0.
\end{align}
\begin{theorem}[Interpolation operator $\mathcal{I}_{K_t}$]\label{thm:It}
The operator $\mathcal{I}_t\colon H^{1}(K_t) \to \mathbb{P}_k(K_t)$ is a linear projection onto $\mathbb{P}_k(K_t)$ satisfying the commuting diagram property
\begin{align}\label{eq:CommutingIt}
\partial_t \mathcal{I}_{K_t} = \Pi_{\mathbb{P}_{k-1}(K_t)} \partial_t.
\end{align}
For all $v\in H^1(K_t)$ the difference $v- \mathcal{I}_{K_t} v\in H^1_0(K_t)$ has zero boundary values and
\begin{align*}
\lVert \partial_t (v - \mathcal{I}_{K_t} v) \rVert_{L^2(K_t)} & = \min_{v_h \in \mathbb{P}_k(K_t)} \lVert \partial_t (v - v_h) \rVert_{L^2(K_t)}. 
\end{align*}
\end{theorem}
\begin{proof}
Let $v\in H^1(K_t)$. Since by definition $v- \mathcal{I}_{K_t} v\in H^1_0(K_t)$, an integration by parts and the definition of $\mathcal{I}_{K_t}^2$ yield for all $w_h\in \mathbb{P}_{k-1}(K_t)$
\begin{align*}
&\int_{K_t} \partial_t (v- \mathcal{I}_{K_t} v)\, w_h\ds = -\int_{K_t} (v- \mathcal{I}_{K_t}v)\, \partial_t w_h \ds\\
&\qquad =  \int_{K_t} b_{K_t} (\mathcal{I}^2_{K_t} v)\, \partial_t w_h \ds -\int_{K_t} (v- \mathcal{I}^1_{K_t} v)\, \partial_t w_h \ds = 0. 
\end{align*}
This proves the commuting diagram property. The commuting diagram property yields the best-approximation property and leads to the projection property.
\end{proof}
By applying the operator everywhere in $\Omega$, the operator extends to a mapping 
\begin{align*}
\mathcal{I}_{K_t} \colon H^1(K_t;H^{-1}(\Omega)) \to \mathbb{P}_k(K_t;H^{-1}(\Omega)).
\end{align*} 
Applying $\mathcal{I}_{K_t}$ on each time cell $K_t\in \tria_t$ leads to the operator $\mathcal{I}_t \colon L^2(\mathcal{J};H^{-1}(\Omega)) \to \mathcal{L}_k^1(\tria_t;H^{-1}(\Omega))$ with
\begin{align}\label{eq:DefItJ}
(\mathcal{I}_t v)|_{K_t} \coloneqq \mathcal{I}_{K_t} v|_{K_t}\qquad\text{for all }v\in H^1(\mathcal{J};H^{-1}(\Omega))\text{ and }K_t\in \tria_t.
\end{align}
\section{Tensor product meshes}\label{sec:interpolOperatorsTensor}
This section introduces interpolation operators for special cylindrical partitions of $Q$, namely tensor product meshes $\tria = \tria_t \otimes \tria_x$ with a partition $\tria_t$ of the time interval $\mathcal{J}$ and a conforming simplicial partition $\tria_x$ of the domain $\Omega$. Such partitions are of special interest since classical time-marching schemes can be seen as a space-time ansatz using such meshes and ansatz spaces $X_h = \mathcal{L}^1_\ell(\tria_t;\mathcal{L}^1_{\ell,0}(\tria_x))$ as well as some specific discretization of the test space $L^2(\mathcal{J};H^1_0(\Omega))$; see for example \cite{UrbanPatera14,Feischl22} for the Crank-Nicolson scheme. 
We introduce and investigate a suitable interpolation operator in the first subsection. The second subsection introduces and investigates an interpolation operator for mixed schemes.
\subsection{Interpolation operator $\mathcal{I}_X^\otimes$}\label{sec:IX}
Due to the tensor product structure of the mesh $\tria = \tria_t\otimes \tria_x$, the discrete space $X_h$ defined in \eqref{eq:DefXh} equals $X_h = \mathcal{L}^1_\ell(\tria_t) \otimes \mathcal{L}^1_{\ell,0}(\tria_x)$.
This allows for the direct application of the operators 
\begin{align*}
\begin{aligned}
&\mathcal{I}_x\colon L^2(\mathcal{J};H^{-1}(\Omega)) \to L^2(\mathcal{J};\mathcal{L}^1_{\ell,0}(\tria_x))&&\text{defined in \eqref{eq:defIxJ}},\\
&\mathcal{I}_t \colon H^1(\mathcal{J};H^{-1}(\Omega)) \to \mathcal{L}^1_k(\tria_t;H^{-1}(\Omega))&&\text{defined in \eqref{eq:DefItJ}}.
\end{aligned}
\end{align*}
More precisely, we set the interpolation operator $\mathcal{I}_X^\otimes \colon X \to X_h$ as the composition
\begin{align}\label{eq:defIX}
\mathcal{I}_X^\otimes  \coloneqq \mathcal{I}_x \circ \mathcal{I}_t = \mathcal{I}_t \circ \mathcal{I}_x.
\end{align}
This operator has the following beneficial properties involving the local mesh sizes $h_x(K) \coloneqq \textup{diam}(K_x)$ and $h_t(K) \coloneqq |K_t|$ for all $K = K_t \times K_x \in \tria$.
\begin{theorem}[Interpolation operator $\mathcal{I}_X^\otimes $]\label{thm:IX}
The operator $\mathcal{I}_X^\otimes $ satisfies for all $v\in X$
\begin{align*}
\lVert \nabla_x (v-\mathcal{I}_X^\otimes  v)\rVert_{L^2(Q)}^2 &\lesssim \sum_{K\in \tria} \min_{v_{x}\in L^2(K_t;\mathcal{L}^1_{\ell,0}(\tria_x))} \lVert \nabla_x (v-v_{x})\rVert_{L^2(K_t;L^2(\omega_{K_x}))}^2 \\
&\quad + \frac{h_t(K)^2}{h_x(K)^4} \min_{v_{t}\in \mathbb{P}_k(K_t;H^{-1}(\omega_{K_x}))} \lVert \partial_t (v-v_{t}) \rVert_{L^2(K_t;H^{-1}(\omega_{K_x}))}^2.  
\end{align*}
Moreover, we have for all $v\in H^1(\mathcal{J};H^{-1}(\Omega))$ the upper bound
\begin{align*}
\lVert \partial_t (v - \mathcal{I}_X^\otimes v)\rVert_{L^2(\mathcal{J};H^{-1}(\Omega))}^2 &\lesssim \sum_{K\in \tria} \min_{\xi_{x}\in L^2(K_t;\mathcal{L}^1_{\ell,0}(\tria_x))} \lVert \partial_t v - \xi_{x} \rVert_{L^2(K_t;H^{-1}(\omega_{K_x}))}^2 \\
&\quad + \sum_{K_t\in \tria_t} \min_{v_{t} \in \mathbb{P}_k(K_t;H^{-1}(\Omega))} \lVert \partial_t (v-v_{t})\rVert_{L^2(K_t;H^{-1}(\Omega))}^2.
\end{align*}
\end{theorem}
\begin{proof}
Let $v\in X$. 
The triangle inequality yields
\begin{align}\label{eq:Splitt}
\lVert \nabla_x (v-\mathcal{I}_X^\otimes v)\rVert_{L^2(Q)} & \leq  \lVert \nabla_x (v - \mathcal{I}_x v)\rVert_{L^2(Q)} +  \lVert \nabla_x \mathcal{I}_x (v-\mathcal{I}_t v)\rVert_{L^2(Q)}.
\end{align}
The approximation properties displayed in Theorem~\ref{thm:Ix} yield for the first addend
\begin{align*}
\lVert \nabla_x (v - \mathcal{I}_x v)\rVert_{L^2(Q)}^2 \eqsim \sum_{K\in \tria} \min_{v_{x}\in L^2(K_t;\mathcal{L}^1_{\ell,0}(\tria_x))} \lVert \nabla_x (v-v_{x})\rVert_{L^2(K_t;L^2(\omega_{K_x}))}^2.
\end{align*}
Due to inverse estimates (Lemma~\ref{lem:inverseEst}) and Theorem~\ref{thm:It} the second addend satisfies
\begin{align}\label{eq:changeOfNorms}
\begin{aligned}
&\lVert \nabla_x \mathcal{I}_x (v-\mathcal{I}_t v)\rVert_{L^2(Q)}^2 = 
\sum_{K\in \tria} \lVert \nabla_x \mathcal{I}_x (v-\mathcal{I}_t v)\rVert_{L^2(K)}^2\\
& \qquad\lesssim \sum_{K\in \tria} h_x(K)^{-4} \lVert v-\mathcal{I}_t v\rVert_{L^2(K_t;H^{-1}(\omega_{K_x}))}^2\\
& \qquad\lesssim \sum_{K\in \tria} \frac{h_t(K)^2}{h_x(K)^4} \min_{v_{t}\in \mathbb{P}_k(K_t;H^{-1}(\Omega))} \lVert \partial_t (v-v_{t}) \rVert_{L^2(K_t;H^{-1}(\omega_{K_x}))}^2.  
\end{aligned}
\end{align}
This proves the first inequality in the theorem.

Let $v\in L^2(\mathcal{J};H^{-1}(\Omega))$. Since $\mathcal{I}_x \partial_t = \partial_t\mathcal{I}_x$, we have
\begin{align*}
\lVert \partial_t( v  -\mathcal{I}_x v) \rVert_{L^2(\mathcal{J};H^{-1}(\Omega))}& \leq \lVert \partial_t v - \mathcal{I}_x \partial_t v\rVert_{L^2(\mathcal{J};H^{-1}(\Omega))}\\
&\quad + \lVert \mathcal{I}_x \partial_t  (v -  \mathcal{I}_t v) \rVert_{L^2(\mathcal{J};H^{-1}(\Omega))}.
\end{align*}
An application of Theorem~\ref{thm:Ix} to the first addend yields
\begin{align*}
\lVert \partial_t v - \mathcal{I}_x \partial_t v\rVert_{L^2(\mathcal{J};H^{-1}(\Omega))}^2 \lesssim \sum_{K\in \tria} \min_{\xi_{x}\in L^2(K_t;\mathcal{L}^1_{\ell,0}(\tria_x))} \lVert \partial_t v - \xi_{x} \rVert_{L^2(K_t;H^{-1}(\omega_{K_x}))}^2.
\end{align*} 
The $H^{-1}(\Omega)$ stability of $\mathcal{I}_x$ and the approximation properties of $\mathcal{I}_t$ yield 
\begin{align*}
 \lVert \mathcal{I}_x \partial_t  (v -  \mathcal{I}_t v) \rVert_{L^2(\mathcal{J};H^{-1}(\Omega))}^2& \lesssim  \lVert \partial_t  (v -  \mathcal{I}_t v) \rVert_{L^2(\mathcal{J};H^{-1}(\Omega))}^2\\
 & = \sum_{K_t\in \tria_t} \min_{v_{t}\in \mathbb{P}_k(K_t;H^{-1}(\Omega))} \lVert \partial_t (v - v_{t}) \rVert_{L^2(K_t;H^{-1}(\Omega))}^2.
\end{align*}
Combining the estimates concludes the proof.
\end{proof}
Due to the continuous embedding $X\hookrightarrow C^0(\mathcal{J};L^2(\Omega))$ in Lemma~\ref{lem:embedding}, we have for all $v\in X$ and $t\in \mathcal{J} = [0,T]$ the upper bound
\begin{align*}
&\lVert v(t) - (\mathcal{I}_X^\otimes v)(t) \rVert_{L^2(\Omega)}\\
&\qquad \lesssim (1+T^{-1})\lVert \nabla_x (v-\mathcal{I}_X^\otimes v)\rVert_{L^2(Q)} + \lVert \partial_t (v-\mathcal{I}_X^\otimes v)\rVert_{L^2(\mathcal{J};H^{-1}(\Omega))}.
\end{align*}
The following result improves this bound. We set the diameters $h_t(K_t) \coloneqq |K_t|$ and $h_x(K_x) \coloneqq \textup{diam}(K_x)$ for all $K_t\in \tria_t$ and $K_x\in \tria_x$.
\begin{theorem}[Interpolation error in $C^0(\mathcal{J};L^2(\Omega))$]
Let $t\in K_t\in \tria_t$ and $v\in X$. Then we have
\begin{align*}
&\lVert (v - \mathcal{I}_X^\otimes v)(t) \rVert_{L^2(\Omega)}^2 \lesssim \min_{v_t\in \mathbb{P}_k(K_t;H^{-1}(\Omega))} \lVert \partial_t  (v-v_t)\rVert^2_{L^2(K_t;H^{-1}(\Omega))} \\
&\qquad+ \sum_{K_x\in \tria_x} \left(1+\frac{h_x(K_x)^2}{h_t(K_t)}\right) \min_{v_{x}\in L^2(K_t;\mathcal{L}^1_{\ell,0}(\tria_x))} \lVert \nabla_x( v - v_{x}) \rVert_{L^2(K_t;L^2(\omega_{K_x}))}^2\\
&\qquad + \sum_{K_x\in \tria_x} \frac{h_t(K_t)}{h_x(K_x)^2} \min_{v_{t} \in \mathbb{P}_k(K_t;H^{-1}(\Omega))} \lVert \partial_t (v-v_{t})\rVert_{L^2(K_t;H^{-1}(\omega_{K_x}))}^2 \\
&\qquad + \sum_{K_x\in \tria_x} \min_{\xi_x \in L^2(K_t;\mathcal{L}^1_{\ell,0}(\tria_x)))} \lVert \partial_t v-\xi_x \rVert_{L^2(K_t;H^{-1}(\omega_{K_x}))}^2.
\end{align*}
\end{theorem}
\begin{proof}
Let $t\in K_t\in \tria_t$ and $v\in X$. Lemma~\ref{lem:embedding} reveals that
\begin{align*}
\lVert (v - \mathcal{I}_X^\otimes v)(t) \rVert_{L^2(\Omega)}^2 &\leq \frac{1}{h_t(K_t)}\, \lVert v - \mathcal{I}_X^\otimes v \rVert_{L^2(K_t;L^2(\Omega))}^2  + \lVert \nabla_x( v - \mathcal{I}_X^\otimes v) \rVert_{L^2(K_t;L^2( \Omega))}^2\\
&\quad + \lVert \partial_t (v - \mathcal{I}_X^\otimes v) \rVert^2_{L^2(K_t;H^{-1}(\Omega))}.
\end{align*}
The arguments in the proof of Theorem~\ref{thm:IX} lead to the bound
\begin{align*}
&\lVert v - \mathcal{I}_X^\otimes v \rVert_{L^2(K_t;L^2(\Omega))}^2 = \sum_{K_x\in \tria_x} \lVert v - \mathcal{I}_X^\otimes v \rVert_{L^2(K_t;L^2(K_x))}^2 \\
&\lesssim \sum_{K_x\in \tria_x}  h_x(K_x)^2 \min_{v_{x}\in L^2(K_t;\mathcal{L}^1_{\ell,0}(\tria_x))} \lVert \nabla_x( v - v_{x}) \rVert_{L^2(K_t;L^2(\omega_{K_x}))}^2\\
&\quad + \sum_{K_x\in \tria_x} \frac{h_t(K_t)^2}{ h_x(K_x)^2} \min_{v_{t} \in \mathbb{P}_k(K_t;H^{-1}(\Omega))} \lVert \partial_t (v-v_{t})\rVert_{L^2(K_t;H^{-1}(\omega_{K_x}))}^2.
\end{align*} 
Combining this estimate with the approximation properties displayed in Theorem~\ref{thm:IX} concludes the proof.
\end{proof}
We conclude this subsection with two remarks.
\begin{remark}[Stability in $L^2(\mathcal{J};H^1_0(\Omega))$]\label{rem:sabL2Hmin}
While the operator $\mathcal{I}_X^\otimes$ is always stable in $H^1(\mathcal{J};H^{-1}(\Omega))$, its (uniform) stability in $X$ requires the parabolic scaling $h_t(K) \eqsim h_x(K)^2$ for all $K\in \tria$. This is due to the change of the norm $\lVert \nabla_x \, \bigcdot \rVert_{L^2(K_t;L^2(K_x))}$ to $\lVert \partial_t\, \bigcdot \rVert_{H^1(K_t;H^{-1}(K_x))}$ in \eqref{eq:changeOfNorms}. It is possible to avoid this change of norms when $\mathcal{I}_t$ is replace by some $L^2$ stable projection operator $\mathcal{I}_t^{L^2}\colon L^2(\mathcal{J}) \to \mathcal{L}^1_k(\tria_t)$ like the Scott-Zhang interpolation operator \cite{ScottZhang90} as done in \cite[Sec.~4.2]{DieningStornTscherpel22}.
Set $(\mathcal{I}_X^\otimes)' \coloneqq \mathcal{I}_x \circ \mathcal{I}_t^{L^2}$ and assume that neighboring time cells $K_t,K_t'\in \tria_t$ are of equivalent size. A similar proof as in Theorem~\ref{thm:IX} leads for all $v\in L^2(\mathcal{J};H^1_0(\Omega))$ to
\begin{align*}
&\lVert \nabla_x (v-(\mathcal{I}_X^\otimes)' v)\rVert_{L^2(Q)}^2\\ 
&\lesssim \sum_{K\in \tria} \min_{v_{x}\in L^2(K_t;\mathcal{L}^1_{\ell,0}(\tria_x))} \lVert \nabla_x (v-v_{x})\rVert_{L^2(K_t;L^2(\omega_{K_x}))}^2 \\
&\quad\qquad + \min_{W_{t}\in \mathcal{L}^1_k(\tria_t;L^2(\Omega;\mathbb{R}^d))} \lVert \nabla_x v- W_{t} \rVert_{L^2(\omega_{K_t};L^2(\omega_{K_x}))}^2.  
\end{align*}
Furthermore, it satisfies for all $v\in H^1(\mathcal{J};H^{-1}(\Omega))$
\begin{align*}
&\lVert \partial_t (v - (\mathcal{I}_X^\otimes)' v)\rVert_{L^2(\mathcal{J};H^{-1}(\Omega))}^2\\
 &\lesssim \sum_{K\in \tria} \min_{\xi_{x}\in L^2(K_t;\mathcal{L}^1_{\ell,0}(\tria_x))} \lVert \partial_t v - \xi_{x} \rVert_{L^2(K_t;H^{-1}(\omega_{K_x}))}^2 \\
&\quad + \sum_{K_t\in \tria_t} \min_{v_{t} \in \mathcal{L}^1_k(\tria_t;H^{-1}(\Omega))} \lVert \partial_t (v-v_{t})\rVert_{L^2(\omega_{K_t};H^{-1}(\Omega))}^2.
\end{align*}
Note that this operator has increased the domain of dependence with respect to the time direction compared to the operator in Theorem~\ref{thm:IX}.
\end{remark}
\begin{remark}[Localization of the $H^1(K_t;H^{-1}(\Omega))$ norm]\label{rem:LoclHmin}
While the interpolation error for $\lVert \nabla_x (v-\mathcal{I}_X^\otimes  v)\rVert_{L^2(Q)}$ decomposes into localized norms, the interpolation error $\lVert \partial_t (v - \mathcal{I}_X^\otimes v)\rVert_{L^2(\mathcal{J};H^{-1}(\Omega))}$ localizes only in time. Lemma~\ref{lem:LocHminOne} shows that it is possible to localize further but at the cost of negative powers of the local mesh size, that is for all $v\in H^1(\mathcal{J};H^{-1}(\Omega))$
\begin{align*}
& \sum_{K_t\in \tria_t} \min_{v_{t} \in \mathbb{P}_k(K_t;H^{-1}(\Omega))} \lVert \partial_t (v-v_{t})\rVert_{L^2(K_t;H^{-1}(\Omega))}^2\\
& \qquad\lesssim \sum_{K\in \tria} h_x(K)^{-2} \min_{v_{t} \in \mathbb{P}_k(K_t;H^{-1}(\omega_{K_x}))} \lVert \partial_t (v-v_{t})\rVert_{L^2(K_t;H^{-1}(\omega_{K_x}))}^2.  
\end{align*}
This upper bound is indeed sharp, as the following consideration shows. 

Suppose there exists with some $s<2$ for all $v\in H^1(\mathcal{J};H^{-1}(\Omega))$ an estimate
\begin{align*}
\lVert \partial_t (v- \mathcal{I}_X^\otimes v)\rVert_{L^2(\mathcal{J};H^{-1}(\Omega))}^2 & \lesssim \sum_{K\in \tria} \min_{v_{x}\in L^2(K_t;\mathcal{L}^1_{\ell,0}(\tria_x))} \lVert \partial_t (v - v_{x}) \rVert_{L^2(K_t;H^{-1}(\omega_{K_x}))}^2\\
& \quad+ h_x(K)^{-s} \min_{v_{t}\in \mathbb{P}_k(K_t;H^{-1}(\omega_{K_x}))} \lVert \partial_t (v - v_{t}) \rVert_{L^2(K_t;H^{-1}(\omega_{K_x}))}^2.
\end{align*}
The estimate holds in particular for functions $w = w_t w_{x}$ with $w_t\in H^1(\mathcal{J})$ and $w_{x} \in \mathcal{L}^1_{\ell,0}(\tria_x)$, that is
\begin{align*}
& \lVert \partial_t (w_t - \mathcal{I}_t w_t) \rVert_{L^2(\mathcal{J})}^2 \lVert w_{x} \rVert_{H^{-1}(\Omega)}^2 = \lVert \partial_t (w - \mathcal{I}_X^\otimes  w)\rVert_{L^2(\mathcal{J};H^{-1}(\Omega))}^2 \\
&\qquad \lesssim \sum_{K\in \tria} h_x(K)^{-s} \min_{w_{h,t}\in \mathbb{P}_k(K_t;H^{-1}(\omega_{K,x}))} \lVert \partial_t (w  - w_{t}) \rVert_{L^2(K_t;H^{-1}(\omega_{K,x}))}^2\\
&\qquad\lesssim \sum_{K\in \tria} h_x(K)^{2-s} \lVert \partial_t w \rVert_{L^2(K_t;L^2(\omega_{K_x}))}^2.
\end{align*}
Hence, we have
\begin{align*}
\lVert \partial_t (w_t - \mathcal{I}_t w_t) \rVert_{L^2(\mathcal{J})}\lVert w_{x} \rVert_{H^{-1}(\Omega)} \lesssim \max_{K\in \tria}\, h_x(K)^{1-s/2} \lVert \partial_t w_t \rVert_{L^2(\mathcal{J})} \lVert w_{x}\rVert_{L^2(\Omega)}.
\end{align*}
This proves convergence of $\lVert \partial_t (w_t - \mathcal{I}_t w_t) \rVert_{L^2(\mathcal{J})}$ independent of the time discretization, which cannot be possible.
\end{remark}

\subsection{Commuting interpolation operator $\mathcal{I}^\otimes_\Lambda$}\label{subsec:ILambda}
Simultaneous space-time minimal residual methods \cite{FuehrerKarkulik19,GantnerStevenson20,GantnerStevenson22,DieningStorn21}
and time marching schemes in mixed form \cite{JohnsonThomee81,BauseRaduKoecher17,KorikachePaquet22} involve
the time-space divergence $\textup{div}\, (v,\tau) \coloneqq \partial_t v + \textup{div}_x \tau$ and the related space
\begin{align*}
\Lambda = \lbrace (v,\tau) \in L^2(\mathcal{J};H^1_0(\Omega)) \times L^2(Q;\mathbb{R}^d)\colon \textup{div}\,(v,\tau) \in L^2(Q)\rbrace.
\end{align*}
We set for all $(v,\tau) \in \Lambda$ the squared norm
\begin{align*}
\lVert (v,\tau) \rVert_\Lambda^2 \coloneqq \lVert \nabla_x v \rVert_{L^2(Q)}^2 + \lVert \tau \rVert_{L^2(Q)}^2 + \lVert \textup{div}\, (v,\tau)\rVert_{L^2(Q)}^2.
\end{align*}
Any function $(v,\tau) \in \Lambda$ satisfies \cite[Lem.~2.1]{GantnerStevenson20}
\begin{align}\label{eq:partTboundedbyLambda}
\lVert \partial_t v \rVert_{L^2(\mathcal{J};H^{-1}(\Omega))} \lesssim \lVert \tau \rVert_{L^2(Q)} + \lVert \textup{div}\, (v,\tau)\rVert_{L^2(Q)} \lesssim \lVert (v,\tau)\rVert_\Lambda.
\end{align}
In particular, we have with $\Sigma \coloneqq L^2(\mathcal{J};H(\textup{div}_x,\Omega))$ and $H(\textup{div}_x,\Omega) \coloneqq \lbrace q \in L^2(\Omega;\mathbb{R}^d)\colon \textup{div}_x q \in L^2(\Omega)\rbrace$ with spacial divergence $\textup{div}_x$ the inclusion
\begin{align*}
\Lambda = \lbrace (v,\tau) \in X \times L^2(Q;\mathbb{R}^d)\colon  \textup{div}\,(v,\tau) \in L^2(Q)\rbrace \subset X \times \Sigma.
\end{align*}
Let $\tria = \tria_t \otimes \tria_x$ be a tensor product mesh with conforming triangulations $\tria_t$ of $\mathcal{J}$ and $\tria_x$ of $\Omega$. Set the Raviart-Thomas finite element space $RT_\ell(\tria_x)$, which reads with identity mapping $\textup{id}\colon \Omega \to \Omega$ 
\begin{align*}
RT_\ell(\tria_x) \coloneqq \lbrace \tau \in H(\textup{div}_x,\Omega)\colon \tau|_{K_x} \in \mathbb{P}_\ell(K_x) + \textup{id} \cdot \mathbb{P}_\ell(K_x;\mathbb{R}^d)\rbrace.
\end{align*}
Let $\Pi_{\mathcal{L}^0_{\ell-1}(\tria_x)}\colon L^2(\Omega)\to \mathcal{L}^0_{\ell-1}(\tria_x)$ be the $L^2(\Omega)$ orthogonal projector onto $\mathcal{L}^0_{\ell-1}(\tria_x)$.
\begin{theorem}[Commuting interpolation operator $\mathcal{I}_{RT}$]\label{thm:Irt}
There exists an interpolation operator $\mathcal{I}_{RT}\colon L^2(\Omega;\mathbb{R}^d) \to RT_\ell(\tria_x)$ with the commuting diagram property
\begin{align*}
\textup{div}_x \mathcal{I}_{RT} \tau = \Pi_{\mathcal{L}^0_{\ell}(\tria_x)} \textup{div}_x \tau\qquad\text{for all }\tau \in H(\textup{div}_x,\Omega).
\end{align*}
Moreover, it has for all $p\in L^2(\Omega;\mathbb{R}^d)$ the approximation property
\begin{align*}
\lVert p - \mathcal{I}_{RT} p \rVert_{L^2(\Omega)}^2 \eqsim \sum_{K_x \in \tria_x} \min_{p_h\in RT_\ell(\tria_x)} \lVert p - p_h \rVert_{L^2(\omega_{K_x})}^2.
\end{align*}
\end{theorem}
\begin{proof}
A suitable operator is investigated for example in \cite[Sec.~23]{ErnGuermond21}.
\end{proof}
We set the discrete subspace 
\begin{align*}
\Sigma_h \coloneqq \mathcal{L}^0_{k-1}(\tria_t;RT_\ell(\tria_x)) = \mathcal{L}^0_{k-1}(\tria_t) \otimes RT_\ell(\tria_x) \subset \Sigma.
\end{align*}
Moreover, we denote the $L^2$ orthogonal projections onto the space of piece-wise polynomials in time  $\mathcal{L}^0_{k-1}(\tria_t)$ and piece-wise polynomials in space $\mathcal{L}^0_\ell(\tria_x)$ by 
\begin{align*}
\Pi_{\mathcal{L}^0_{k-1}(\tria_t)} &\colon L^2(Q;\mathbb{R}^r)\to \mathcal{L}^0_{\ell-1}(\tria_t;L^2(\Omega;\mathbb{R}^r))\qquad\text{with }r\in \lbrace 1,d\rbrace,\\
\Pi_{\mathcal{L}^0_\ell(\tria_x)}& \colon L^2(Q)\to L^2(\mathcal{J};\mathcal{L}^0_\ell(\tria_x)).
\end{align*}
We set the interpolation operator 
\begin{align*}
\mathcal{I}^\otimes_\Sigma \coloneqq \Pi_{\mathcal{L}^0_{k-1}(\tria_t)} \circ  \mathcal{I}_{RT} \colon L^2(Q) \to \Sigma_h.
\end{align*}
\begin{theorem}[Commuting interpolation operator $\mathcal{I}^\otimes_\Sigma$]\label{thm:ISigma}
We have the commuting diagram property
\begin{align*}
\textup{div}_x\mathcal{I}^\otimes_\Sigma = \Pi_{\mathcal{L}^0_{k-1}(\tria_t)} \Pi_{\mathcal{L}^0_{\ell}(\tria_x)} \textup{div}_x.
\end{align*}
Moreover, we have for all $p \in L^2(Q;\mathbb{R}^d)$ the approximation property
\begin{align*}
\lVert p - \mathcal{I}^\otimes_\Sigma p\rVert_{L^2(Q)}^2 & \eqsim \sum_{K\in \tria} \min_{p_{x} \in L^2(K_t;RT_\ell(\tria_x))} \lVert p - p_{x} \rVert_{L^2(K_t;L^2(\omega_{K_x}))}^2\\
&\qquad\qquad + \min_{p_t \in \mathbb{P}_\ell(K_t;L^2(\omega_{K_x};\mathbb{R}^d))} \lVert p - p_{t} \rVert_{L^2(K_t;L^2(K_x))}^2.
\end{align*} 
\end{theorem}
\begin{proof}
Let $p \in L^2(Q;\mathbb{R}^d)$. The triangle inequality yields
\begin{align*}
\lVert p - \mathcal{I}^\otimes_\Sigma p\rVert_{L^2(Q)} &\leq \lVert p - \Pi_{\mathcal{L}^0_{k-1}(\tria_t)} p \rVert_{L^2(Q)} + \lVert \Pi_{\mathcal{L}^0_{k-1}(\tria_t)} (p - \mathcal{I}_{RT}p)\rVert_{L^2(Q)}\\
&\leq \lVert p - \Pi_{\mathcal{L}^0_{k-1}(\tria_t)} p \rVert_{L^2(Q)} + \lVert p - \mathcal{I}_{RT} p \rVert_{L^2(Q)}.
\end{align*}
The approximation properties of the semi-discrete operators $\Pi_{\mathcal{L}^0_{k-1}(\tria_t)}$ and $\mathcal{I}_{RT}$ lead to the approximation property in the theorem. Theorem~\ref{thm:Irt} implies the commuting diagram property.
\end{proof}
We set the discrete subspace $\Lambda_h \coloneqq X_h \otimes \Sigma_h$. By exploiting the tensor product structure we can define for each $(v,\tau) \in \Lambda$ an interpolation $(\mathcal{I}_X^\otimes v,\mathcal{I}^\otimes_\Sigma \tau) \in \Lambda_h$ with good approximation properties. 
In fact, a similar interpolation operator has been suggested in \cite{GantnerStevenson22}. We modify this ansatz to achieve additionally a commuting diagram property. The modification involves the application of the inverse Laplacian $(-\Delta_x)^{-1}\colon L^2(\mathcal{J};H^{-1}(\Omega)) \to L^2(\mathcal{J};H^1_0(\Omega))$ everywhere in time defined for all $\xi \in L^2(\mathcal{J};H^{-1}(\Omega))$ as solution operator to 
\begin{align*}
\langle \nabla_x (-\Delta_x)^{-1} \xi , \nabla_x w\rangle_\Omega = \langle \xi,w\rangle_\Omega\qquad\text{for all }w\in L^2(\mathcal{J};H^1_0(\Omega)).
\end{align*}
We set for all $(v,\tau)\in \Lambda$ the interpolation operator $\mathcal{I}^\otimes_\Lambda\colon \Lambda \to \Lambda_h$ as
\begin{align*}
\mathcal{I}_{\Lambda}(v,\tau) \coloneqq (\mathcal{I}_X^\otimes  v, \mathcal{I}_2(v,\tau))\text{ with }\mathcal{I}_2(v,\tau) = \mathcal{I}^\otimes_\Sigma \big(\tau-\nabla_x (-\Delta_x)^{-1} \partial_t (v - \mathcal{I}_X^\otimes  v)\big).
\end{align*}
\begin{theorem}[Commuting diagram property and approximablity of $\mathcal{I}^\otimes_\Lambda$]\label{thm:ILambda}
The projection $\mathcal{I}^\otimes_\Lambda$ onto $\Lambda_h$ commutes in the sense that for all $(v,\tau)\in \Lambda$
\begin{align*}
\textup{div}\, \mathcal{I}^\otimes_\Lambda(v,\tau) = \Pi_{\mathcal{L}^0_{k-1}(\tria_t)} \Pi_{\mathcal{L}^0_\ell(\tria_x)} \, \textup{div}\, (v,\tau).
\end{align*}
Moreover, we control for all $(v,\tau)\in X \times \Sigma$ the interpolation error by
\begin{align}\label{eq:bestApxIu}
\lVert \tau - \mathcal{I}_2(v,\tau)\rVert_{L^2(Q)} \lesssim \lVert \tau - \mathcal{I}^\otimes_\Sigma \tau \rVert_{L^2(Q)} + \lVert \partial_t (v-\mathcal{I}_X^\otimes  v)\rVert_{L^2(\mathcal{J};H^{-1}(\Omega))}.
\end{align}
\end{theorem}
\begin{proof} Let $(v,\tau) \in \Lambda$.
The commuting diagram property of $\mathcal{I}^\otimes_\Sigma$ (Theorem~\ref{thm:ISigma}) and $- \textup{div}_x \nabla_x (-\Delta_x)^{-1} \partial_t (v -\mathcal{I}_X^\otimes  v) = \partial_t (v -\mathcal{I}_X^\otimes  v)$ yield
\begin{align*}
\textup{div}_x \mathcal{I}_2(v,\tau)& = \textup{div}_x \mathcal{I}^\otimes_\Sigma\big( \tau + \partial_t (v - \mathcal{I}_X^\otimes  u)\big)\\
& = \Pi_{\mathcal{L}^0_{k-1}(\tria_t)} \Pi_{\mathcal{L}^0_\ell(\tria_x)}
 \textup{div}_x \big( \tau - \nabla_x (-\Delta_x)^{-1} \partial_t (v -\mathcal{I}_X^\otimes  v)\big)\\
& = -\partial_t \mathcal{I}_X^\otimes  v + \Pi_{\mathcal{L}^0_{k-1}(\tria_t)} \Pi_{\mathcal{L}^0_\ell(\tria_x)} (\partial_t v + \textup{div}_x \tau).
\end{align*}
Hence, we have the commuting diagram property
\begin{align*}
\textup{div}\, \mathcal{I}^\otimes_\Lambda(v,\tau)& = \partial_t \mathcal{I}_X^\otimes  v + \textup{div}_x \mathcal{I}_2(v,\tau) = \Pi_{\mathcal{L}^0_{k-1}(\tria_t)} \Pi_{\mathcal{L}^0_\ell(\tria_x)} \textup{div}\, (v,\tau).
\end{align*}
The approximation property follows from an application of the triangle inequality and the $L^2$ stability of $\mathcal{I}_\Sigma$, that is, for all $(v,\tau) \in X \times \Sigma$
\begin{align*}
\lVert \tau - \mathcal{I}_2(v,\tau) \rVert_{L^2(Q)} & \lesssim \lVert \tau - \mathcal{I}^\otimes_\Sigma \tau \rVert_{L^2(Q)} + \lVert \nabla_x (-\Delta_x)^{-1} \partial_t (v -\mathcal{I}_X^\otimes  v) \rVert_{L^2(Q)}\\
& = \lVert \tau - \mathcal{I}^\otimes_\Sigma \tau \rVert_{L^2(Q)} + \lVert \partial_t (v -\mathcal{I}_X^\otimes  v) \rVert_{L^2(\mathcal{J};H^{-1}(\Omega))}.\qedhere
\end{align*}
\end{proof}
\begin{remark}[Smoothing rough-right hand sides]
The papers \cite{FuehrerHeuerKarkulik21,DieningStornTscherpel22,Fuehrer22} suggest smoothing of the right-hand side in least-squares and mixed formulations for the Poisson model problem to conclude optimal rates of convergence even with right-hand sides in $H^{-1}(\Omega)$. The key in the proof are suitable properties of the smoothing operator and the commuting diagram property of the operator $\mathcal{I}_{RT}$. Using the commuting diagram property of the operator $\mathcal{I}_\Lambda$ and using a suitable smoother (which results from the composition of $\Pi_{\mathcal{L}^0_k(\tria_t)}$ and the smoother for the Poisson model problem in space) lead to the same results for least-squares and mixed schemes for the heat equation.
\end{remark}
\section{Irregular meshes}\label{sec:IXirreg}
In this section we introduce an operator $\mathcal{I}_X\colon X \to X_h$ with locally in space-time refined underlying triangulation $\tria$. Such local refinements lead to irregular partitions as for example displayed in Figure~\ref{fig:BadMesh}. 
In order to have some local support of the nodal basis functions $(\phi_j)_{j\in \mathcal{N}} \subset X_h$ where $\mathcal{N}$ denotes the set of degrees of freedom in $X_h$, we need additional assumptions like the 1-irregular rule in \cite{GantnerStevenson22}. 
To avoid technicalities, we do not discuss the impact of these properties and rather state the following assumption.
\begin{itemize}
\item \textit{(Shape regularity)} All $d$-simplices $K_x$ with $K = K_t\times K_x \in \tria$ are shape regular.
\item \textit{(Local grading)} Let $K\in \tria$ and set $\nodes(K) \coloneqq \lbrace j\in \nodes \colon K \subset \textup{supp}(\phi_j) \rbrace$. We define the patch $\omega_K = \omega_{K,t} \times \omega_{K,x} \supset \bigcup_{j\in \nodes(K)} \textup{supp}(\phi_j)$ as the smallest cylinder that contains the support of all basis functions $\phi_j$ with $j\in \nodes(K)$. We assume that simplices $K'=K_t'\times K_x' \in \tria$ with $K' \subset \omega_K$ are of equivalent size in the sense that $|K'_t|\eqsim |K_t| \eqsim |\omega_{K,t}|$ and $|K'_x|\eqsim |K_x| \eqsim |\omega_{K,x}|^{1/d}$.
\end{itemize}  
Let $\Pi_K\colon X \to \mathbb{P}_k(K_t;\mathbb{P}_\ell(K_x))$ denote the $L^2(K)$ orthogonal projection onto the space $\mathbb{P}_k(K_t;\mathbb{P}_\ell(K_x))$ for all $K = K_t \times K_x\in \tria$. 
\begin{lemma}[Operator $\Pi_K$]\label{lem:PiK}
Let $K\in \tria$ and $v \in X$. We have
\begin{align*}
\lVert \nabla_x (v - \Pi_K v)\rVert_{L^2(K)}& \eqsim \min_{v_x \in L^2(K_t;\mathbb{P}_\ell(K_x))} \lVert \nabla_x (v - v_x)\rVert_{L^2(K_t;L^2(K_x))}\\
&\quad + \min_{V_t \in \mathbb{P}_k(K_t;L^2(K_x;\mathbb{R}^d))} \lVert \nabla_x v- V_t \rVert_{L^2(K_t;L^2(K_x))},\\
\lVert \partial_t (v - \Pi_K v)\rVert_{L^2(K)} &\eqsim \min_{\xi_x\in L^2(K_t;\mathbb{P}_{\ell}(K_x))} \lVert \partial_t v - \xi_x \rVert_{L^2(K_t;L^2(K_x))} \\
& \quad + \min_{v_t \in \mathbb{P}_k(K_t;L^2(K_x))} \lVert\partial_t( v - v_t)\rVert_{L^2(K_t;L^2(K_x))}.
\end{align*}
\end{lemma}
\begin{proof}
Let $v\in X$ and let $K = K_t\times K_x\in \tria$. Recall the $L^2(K_t)$ orthogonal projector $\Pi_{\mathbb{P}_k(K_t)}$ defined in \eqref{eq:DefPit}. Let $\Pi_{\mathbb{P}_\ell(K_x)}$ denote the $L^2(K_x)$ orthogonal projection onto $\mathbb{P}_{\ell}(K_x)$.
Since $\Pi_K v = \Pi_{\mathbb{P}_k(K_t)} \Pi_{\mathbb{P}_\ell(K_x)} v$, we have
\begin{align*}
&\lVert \nabla_x ( v - \Pi_K v)\rVert_{L^2(K)}\\
&\qquad \leq \lVert \nabla_x (v - \Pi_{\mathbb{P}_k(K_t)} v) \rVert_{L^2(K)} +  \lVert \Pi_{\mathbb{P}_k(K_t)}\nabla_x (v - \Pi_{\mathbb{P}_\ell(K_x)} v) \rVert_{L^2(K)}\\
&\qquad\leq \lVert \nabla_x (v - \Pi_{\mathbb{P}_k(K_t)} v) \rVert_{L^2(K)} +  \lVert \nabla_x (v - \Pi_{\mathbb{P}_\ell(K_x)} v) \rVert_{L^2(K)}.
\end{align*}
Using the approximation properties of the semi-discrete projection operators leads to the first estimate in the lemma. Similar arguments  yield due to the identity $\Pi_K v = \Pi_{\mathbb{P}_\ell(K_x)} \Pi_{\mathbb{P}_k(K_t)}  v$ the second estimate.
\end{proof}
We assign to each degree of freedom $j\in \nodes$ a simplex $K(j)\in \tria$ with $j\in K$ and set the operator $\mathcal{I}_X \colon X \to X_h$ with
\begin{align*}
(\mathcal{I}_X v)(j) \coloneqq (\Pi_{K(j)} v)(j)\qquad\text{for all }v\in X\text{ and }j\in \nodes.
\end{align*}
\begin{theorem}[Interpolation operator $\mathcal{I}_X$]\label{thm:IXirreg}
The operator $\mathcal{I}_X$ is a projection onto $X_h$ that satisfies for all $v\in X$ 
\begin{align*}
&\lVert \nabla_x (v - \mathcal{I}_X v)\rVert_{L^2(K)} \lesssim \min_{v_h \in X_h} \lVert \nabla_x (v - v_h)\rVert_{L^2(\omega_K)} + \frac{h_t(K)}{h_x(K)}\, \lVert \partial_t v- \partial_t v_h \rVert_{L^2(\omega_K)}.
\end{align*}
Moreover, we have for all $v\in X$ 
\begin{align*}
\lVert \partial_t (v - \mathcal{I}_X v)\rVert_{L^2(K)}& \lesssim \min_{v_h \in X_h} \lVert\partial_t (v - v_h) \rVert_{L^2(\omega_K)}
+ \frac{h_x(K)}{h_t(K)} \lVert \nabla_x v - \nabla_x v_h\rVert_{L^2(\omega_K)}.
\end{align*}
\end{theorem}
\begin{proof}
The projection property of $\mathcal{I}_X$ follows directly by its definition.
Let $v\in X$, $v_h \in X_h$, and $K = K_t\times K_x\in \tria$. Since $(\Pi_K v_h - \mathcal{I}_X v_h)|_K = 0$, we have with $\delta \coloneqq v - v_h$
\begin{align}\label{eq:Proofasdsadsa}
\lVert \nabla_x( v - \mathcal{I}_X v) \rVert_{L^2(K)} \leq \lVert\nabla_x( v - \Pi_K v)\rVert_{L^2(K)} + \lVert \nabla_x(\Pi_K \delta - \mathcal{I}_X \delta) \rVert_{L^2(K)}.
\end{align}
Let $\nodes_\textup{loc}(K)$ denote the degrees of freedom in $\mathbb{P}_{k}(K_t)\otimes \mathbb{P}_{\ell}(K_x)$ with associated basis functions $
b_{\gamma} = b_{\gamma,t}b_{\gamma,x}$, where $b_{\gamma,t} \in \mathbb{P}_k(K_t)$ and $b_{\gamma,x}\in \mathbb{P}_{\gamma}(K_x)$ are such that $b_{\gamma}(\beta) = \delta_{\gamma,\beta}$ for all $\gamma,\beta\in \nodes_\textup{loc}(K)$.
Then there exists a dual basis $b_{\gamma}^* = b_{\gamma,t}^*b_{\gamma,x}^*$ with  
$b_{\gamma,t}^* \in \mathbb{P}_k(K_t)$ and $b_{\gamma,x}^*\in \mathbb{P}_{\ell}(K_x)$ such that 
\begin{align*}
\Pi_K \delta = \sum_{\gamma \in \nodes_\textup{loc}(K)} \langle \delta , b^*_{\gamma}\rangle_Q b_\gamma.
\end{align*}
Hence, we have
\begin{align}\label{eq:ProofGradssad}
&\lVert\nabla_x( \Pi_K \delta - \mathcal{I}_X \delta) \rVert_{L^2(K)} \leq \sum_{\gamma \in \nodes_\textup{loc}(K)} |\langle \delta , b_\gamma^* \rangle_Q  - (\mathcal{I}_X \delta)(\gamma) |\, \lVert   \nabla_x b_\gamma \rVert_{L^2(K)}.
\end{align}
The values of $\mathcal{I}_X\delta$ at the local degree of freedom $\gamma \in \nodes_\textup{loc}(K)$ read as follows. 
\begin{itemize}
\item If $\gamma$ is not on the boundary $\mathcal{J}\times \partial \Omega$, the value of $\mathcal{I}_X$ at $\gamma$ depends on the values of $\mathcal{I}_X$ at some degrees of freedom $(j_\gamma^m)_{m = 1}^{N_{\gamma}}\subset \nodes(K)$ with some uniformly bounded number $N_{\gamma}\in \mathbb{N}$. More precisely, there exist coefficients $\alpha^m_{\gamma} \in \mathbb{R}$ with
\begin{align}\label{eq:Proofasdasd}
(\mathcal{I}_X w)(\gamma) = \sum_{m=1}^{N_{\gamma}} \alpha^m_{\gamma} (\mathcal{I}_Xw)(j_{\gamma}^m)\qquad\text{for all }w\in X.
\end{align}
For each basis function $j_{\gamma}^m \in \nodes(K)$ there exist by definition of $\mathcal{I}_X$ dual weight functions $(b_{\gamma}^m)^* = (b_{\gamma,t}^m)^*(b_{\gamma,x}^m)^*$ with $(b_{\gamma,t}^m)^*\in \mathbb{P}_k(K_t(j_\gamma^m))$ and $(b_{\gamma,x}^m)^*\in \mathbb{P}_\ell(K_x(j_\gamma^m))$ such that
\begin{align*}
(\mathcal{I}_X w)(\gamma) = \sum_{m=1}^{N_{\gamma}} \alpha^m_{\gamma} \langle w, (b_{\gamma}^m)^* \rangle_Q \qquad\text{for all }w\in X.
\end{align*}
\item If the local degree of freedom $\gamma$ is on the boundary $\mathcal{J} \times \partial \Omega$, we have $(\mathcal{I}_X \delta)(\gamma) = 0$ and set $N_\gamma \coloneqq 0$.
\end{itemize}
\textit{Case 1 (No dofs on boundary).} We suppose that a local degree of freedom $\gamma\in \nodes_\textup{loc}(K)$ is neither on the boundary $\mathcal{J} \times \partial \Omega$ nor it depends on some degree of freedom $j\in \nodes(K)$ on the boundary. Then \eqref{eq:Proofasdasd} as well as the fact that $\mathcal{I}_X$ preserves constant functions (away from the boundary) and $\sum_{j \in \nodes(K)} \phi_j|_K= 1$ yield 
\begin{align}\label{eq:Proofasdsadsaadaa}
\sum_{\gamma=1}^{N_{\gamma}} \alpha^\ell_{\gamma} = 1.
\end{align}
We set the integral means $\langle \delta \rangle_{\omega_{K,t}} \coloneqq \dashint_{\omega_{K,t}} \delta \ds$ and $\langle \delta \rangle_{\omega_{K,x}} \coloneqq \dashint_{\omega_{K,x}} \delta \dx$.
Combining \eqref{eq:Proofasdsadsaadaa} with $\langle 1 ,b_{\gamma,t}^*- (b^m_{\gamma,t})^* \rangle_\mathcal{J} = 0 = \langle 1 , (b^m_{\gamma,x})^* - b_{\gamma,x}^* \rangle_\Omega$ and scaling arguments yield
\begin{align}
\begin{aligned}
&|\langle \delta , b_\gamma^* \rangle_Q  - (\mathcal{I}_X \delta)(\gamma)|\, \lVert   \nabla_x b_\gamma \rVert_{L^2(K)} = \Big|\sum_{m=1}^{N_\gamma} \alpha^\ell_{\gamma}  \langle \delta , b_\gamma^* - (b_{\gamma}^m)^* \rangle_Q\Big|\,  \lVert   \nabla_x b_\gamma \rVert_{L^2(K)} \\
&\quad = \Big|\sum_{m=1}^{N_\gamma} \alpha^\ell_{\gamma} \langle \delta - \langle \delta \rangle_{\omega_{K,t}}, b_{\gamma,x}^* (b_{\gamma,t}^*- (b^m_{\gamma,t})^*)\rangle_Q \Big|\,  \lVert   \nabla_x b_\gamma \rVert_{L^2(K)}\\
&\quad \quad + \Big| \sum_{m=1}^{N_\gamma} \alpha^\ell_{\gamma} \langle \delta - \langle \delta \rangle_{\omega_{K,x}}, (b_{\gamma,x}^*-(b_{\gamma,x}^m)^* )(b^m_{\gamma,t})^*)\rangle_Q\Big| \,  \lVert   \nabla_x b_\gamma \rVert_{L^2(K)}\\
&\quad \lesssim \frac{h_t(K)}{h_x(K)} \lVert \partial_t \delta \rVert_{L^2(K)} + \lVert \nabla_x \delta \rVert_{L^2(K)}\qquad\text{for all }\gamma \in \nodes_\textup{loc}(K).
\end{aligned}
\end{align}
\textit{Case 2 (dof on boundary).} Suppose that a degree of freedom $j \in \nodes(K)$ is on the boundary $\mathcal{J} \times \partial \Omega$. Then the patch $\omega_K$ shares a face with the boundary and so  scaling arguments and Friedrichs' inequality lead for all $\gamma \in \nodes_\textup{loc}(K)$ to
\begin{align}\label{eq:ProofEnd}
\begin{aligned}
|\langle \delta , b_\gamma^* \rangle_Q  - (\mathcal{I}_X \delta)(\gamma)|\, \lVert   \nabla_x b_\gamma \rVert_{L^2(K)} &\leq  \Big| \Big\langle \delta , b_\gamma^* - \sum_{m=1}^{N_\gamma} \alpha^\ell_{\gamma}  (b_{\gamma}^m)^* \Big\rangle_Q\Big|\,  \lVert   \nabla_x b_\gamma \rVert_{L^2(K)} \\
& \lesssim \lVert \nabla_x \delta \rVert_{L^2(K)}. 
\end{aligned}
\end{align}
Combining \eqref{eq:Proofasdsadsa}--\eqref{eq:ProofEnd} and Lemma~\ref{lem:PiK} leads to the first bound in the lemma. The second follows similarly.  
\end{proof}
We proceed with a comparison of the interpolation error estimates on tensor product and irregular meshes. Let $K \in \tria$ be a time-space cell with $h_t \coloneqq h_t(K)$ and $h_x \coloneqq h_x(K)$ in a tensor product mesh (if we apply $\mathcal{I}^\otimes_X$) or in an irregular mesh. Due to Theorem~\ref{thm:paraPoincare}, \ref{thm:IX}, and \ref{thm:IXirreg} we have the stability and approximation properties
\begin{align}\label{eq:Convergenceasd}
\begin{aligned}
&\lVert \nabla_x( v - \mathcal{I} v)\rVert_{L^2(K)}\\
&\quad \lesssim \begin{cases}
\lVert \nabla_x v \rVert_{L^2(K)} + \frac{h_t}{h_x^2} \lVert \partial_t v\rVert_{L^2(K_t;H^{-1}(\omega_{K,x}))}&\text{for }\mathcal{I} = \mathcal{I}_X^\otimes,\\
h_x \lVert \nabla^2_x v \rVert_{L^2(K)} + \frac{h^2_t}{h_x^2} \lVert \partial_t^2 v\rVert_{L^2(K_t;H^{-1}(\omega_{K,x}))}&\text{for }\mathcal{I} = \mathcal{I}_X^\otimes,\\
\lVert \nabla_x v \rVert_{L^2(K)} + \frac{h_t}{h_x} \lVert \partial_t v\rVert_{L^2(\omega_K)}&\text{for }\mathcal{I} \in \lbrace \mathcal{I}_X^\otimes,\mathcal{I}_X\rbrace,\\
h_x\lVert \nabla^2_x v \rVert_{L^2(K)} + \frac{h_t}{h_x} \lVert \partial_t v\rVert_{L^2(\omega_K)}&\text{for }\mathcal{I} \in \lbrace \mathcal{I}_X^\otimes,\mathcal{I}_X\rbrace,\\
h_x \lVert \nabla^2_x v \rVert_{L^2(K)} + \frac{h^2_t}{h_x} \lVert \partial_t^2 v\rVert_{L^2(\omega_K)}&\text{for }\mathcal{I} \in \lbrace \mathcal{I}_X^\otimes,\mathcal{I}_X\rbrace,\\
h_x \lVert \nabla^2_x v \rVert_{L^2(K)} + h_t \lVert \partial_t \nabla_x v\rVert_{L^2(\omega_K)}&\\
\hphantom{h_x \lVert \nabla^2_x v \rVert_{L^2(K)}} + \frac{h^2_t}{h_x^2} \lVert \partial_t^2 v\rVert_{L^2(K_t;H^{-1}(\omega_{K,x}))}&\text{for }\mathcal{I} \in \lbrace \mathcal{I}_X^\otimes,\mathcal{I}_X\rbrace.
\end{cases} 
\end{aligned}
\end{align}
Despite a smaller domain of dependency with respect to time (neglected in the comparison above), the advantages of the operator $\mathcal{I}_X^\otimes$ are restricted to stability properties in $X$ rather than in 
$L^2(\mathcal{J};H^1_0(\Omega)) \cap H^1(\mathcal{J};L^2(\Omega))$. However, under reasonable regularity assumptions like in \eqref{eq:Reg1} and \eqref{eq:Reg2} both operators lead to the same approximation properties.  These properties suggest the following mesh scalings.
\begin{itemize}
\item If we only have \eqref{eq:Reg1}, the results suggest the parabolic scaling $h_t \eqsim h_x^2$.
\item If we only have \eqref{eq:Reg1} and \eqref{eq:Reg2}, the results suggest the scaling $h_t \eqsim h_x^{3/2}$.
\item If we additionally have $\partial_t^2 v\in L^2(Q)$, the results suggest the scaling $h_t \eqsim h_x$.
\end{itemize}
While Theorem~\ref{thm:IX} investigates the error $\partial_t (v - \mathcal{I}_X^\otimes v)$ in the $L^2(\mathcal{J};H^{-1}(\Omega))$ norm, Theorem~\ref{thm:IXirreg} investigates the error $\partial_t (v - \mathcal{I}_X v)$ in the $L^2(Q)$ norm.
Notice however, that similar arguments as in the proof of Theorem~\ref{thm:IX} allow us to conclude upper bounds for the $L^2(Q)$ norm of the interpolation error $\partial_t (v - \mathcal{I}_X^\otimes v)$ as well. This leads to the following comparison for $\mathcal{I} \in \lbrace \mathcal{I}_X^\otimes,\mathcal{I}_X\rbrace$, where the values in brackets are solely needed if $\mathcal{I} = \mathcal{I}_X$:
\begin{align*}
&\lVert \partial_t (v-\mathcal{I}v)\rVert_{L^2(K)}\\
& \quad\lesssim \begin{cases}
\lVert \partial_t v \rVert_{L^2(\omega_K)}& \big(+\frac{h_x}{h_t} \lVert \nabla_x v \rVert_{L^2(\omega_K)}\big),\text{ or}\\
h_x \lVert \partial_t \nabla_x v \rVert_{L^2(\omega_K)} + \frac{h_t}{h_x} \lVert \partial_t^2 v \rVert_{L^2(K_t;H^{-1}(\omega_{K,x}))} &\big(+\frac{h^2_x}{h_t} \lVert \nabla^2_x v \rVert_{L^2(\omega_K)}\big),\text{ or} \\
h_x \lVert \partial_t \nabla_x v \rVert_{L^2(\omega_K)} + h_t \lVert \partial_t^2 v \rVert_{L^2(\omega_K)} &\big(+\frac{h_x^2}{h_t} \lVert \nabla_x^2 v \rVert_{L^2(\omega_K)}\big).
\end{cases}
\end{align*}
Under the regularity assumptions in \eqref{eq:Reg1} and \eqref{eq:Reg2} the estimates show for both operators a reduced rate of convergence compared to \eqref{eq:Convergenceasd}. This can be expected, since we investigate the error with respect to the stronger $L^2(Q)$ norm. 
The combination of the error estimates with the regularity properties in \eqref{eq:Reg1} and \eqref{eq:Reg2} suggests a scaling $h_t \eqsim h_x^{3/2}$. If we grade the mesh too strongly, for example $h_t \eqsim h_x^2$, the operator $\mathcal{I}_X$ experiences, unlike the operator $\mathcal{I}_X^\otimes$, stability issues due to the terms 
\begin{align}\label{eq:OptiBadBounds}
 \frac{h_x}{h_t} \lVert \nabla_x v \rVert_{L^2(K)}
\qquad\text{or}\qquad \frac{h_x^2}{h_t} \lVert \nabla_x^2 v \rVert_{L^2(K)}.
\end{align}
Notice that unlike for operators on tensor meshes, such terms which do not depend on the time derivative $\partial_t v$ must occur in bounds for the interpolation error $\partial_t (v -\mathcal{I}_X v)$, since on irregular meshes the interpolated function $\mathcal{I}_X v$ might vary in time even so $v$ might be constant in time, that is, the property $\partial_t v = 0$ does in general not imply $\partial_t \mathcal{I}_X v = 0$. 
The following remark investigates this aspect in more detail.
\begin{remark}[Parabolic scaling vs.~ local refinements]\label{rem:paraVsLocal}
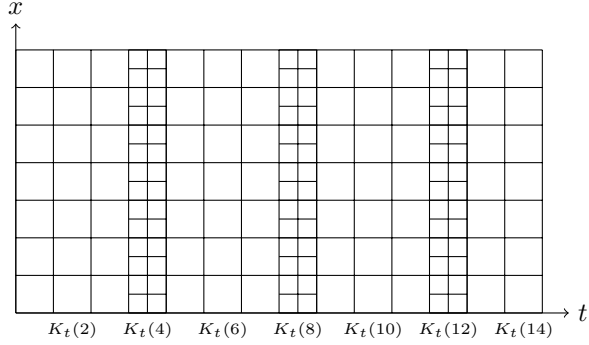
\begin{figure}
\begin{tikzpicture}[scale = 3.5]
\draw [step=1/7] (0.0,0.0) grid (2,1);
\node[right]  at (2.1, 0)   (a) {$t$};
\node[above]  at (0,1.1)   (b) {$x$};
\draw [->] (0.0,0.0) -- (a);
\draw [->] (0.0,0.0) -- (b);
\draw [step=1/14] (3/7,0.0) grid (4/7,1);
\draw [step=1/14] (7/7,0.0) grid (8/7,1);
\draw [step=1/14] (11/7,0.0) grid (12/7,1);
\node[below]  at (7/14,0)   (b) {\tiny $K_t(4)$};
\node[below]  at (3/14,0)   (b) {\tiny $K_t(2)$};
\node[below]  at (11/14,0)   (b) {\tiny $K_t(6)$};
\node[below]  at (15/14,0)   (b) {\tiny $K_t(8)$};
\node[below]  at (19/14,0)   (b) {\tiny $K_t(10)$};
\node[below]  at (23/14,0)   (b) {\tiny $K_t(12)$};
\node[below]  at (27/14,0)   (b) {\tiny $K_t(14)$};
\end{tikzpicture}
\caption{Locally refined meshes}\label{fig:BadMesh}
\end{figure}
Let $\mathcal{I}\colon X \to X_h$ be some locally defined interpolation operator with first order ansatz space $X_h = X_h^{1,1}$ and basis functions $(\phi_j)_{j\in \nodes}$. For simplicity we assume that the operator has weights $\phi_{j}^* \in X$ with $\textup{supp}(\phi_j^*) \subset \textup{supp}(\phi_j)$ and
\begin{align*}
\mathcal{I} v = \sum_{j\in \nodes} \langle v, \phi_{j}^*\rangle \phi_j\qquad\text{for all }v\in X.
\end{align*}
Moreover, we assume that these weights solely depend on the shape of the element patch. Let the underlying mesh result from refining a uniform tensor mesh $\tria_t \otimes \tria_x$ with $0<h_t = |K_t|$ for all $K_t\in \tria_t$ and $0<h_x = \textup{diam}(K_x)$ for all $K_x\in \tria_x$ in every fourth time interval $K_t(4),K_t(8),K_t(12),\dots \in \tria_t$ as depicted in Figure~\ref{fig:BadMesh}.
We can find a function $v\in X$ with $\partial_t v = 0$ such that $(\mathcal{I} v)|_{K_t(4m-2)\times \Omega} = 0$ equals zero on every $(4m-2)$-th time interval in $\tria_t$ with $m\in \mathbb{N}$ and $\phi(j) = 1$ for all degrees of freedom $j\in \nodes$ inside the refined area, that is inside $j\in \textup{int}(K_t(4m) \times \Omega)$. 
Scaling arguments lead to $
\lVert \nabla_x  v\rVert_{L^2(Q)} \eqsim h_x^{-1}$ and $\lVert \nabla^2_x  v\rVert_{L^2(Q)} \eqsim h_x^{-2}$.
By definition we have $\lVert \partial_t \mathcal{I}v\rVert_{L^2(\mathcal{J};H^{-1}(\Omega))} \eqsim \lVert \partial_t \mathcal{I}v\rVert_{L^2(Q))} \eqsim h_t^{-1}$.
Thus, the interpolation error reads
\begin{align*}
\lVert \partial_t(v- \mathcal{I}v)\rVert_{L^2(\mathcal{J};H^{-1}(\Omega))} \eqsim \lVert \partial_t (v-\mathcal{I}v) \rVert_{L^2(Q)} \eqsim \frac{h_x}{h_t} \lVert \nabla_x  v\rVert_{L^2(Q)}  \eqsim \frac{h^2_x}{h_t} \lVert \nabla_x  v\rVert_{L^2(Q)}.
\end{align*}
In this regard the terms in \eqref{eq:OptiBadBounds} cannot be avoided in interpolation error estimates for the time derivative on irregular meshes.
\end{remark}
\section{Conclusion}
This paper introduces interpolation operators and investigates their stability and (localized) approximation properties displayed in Theorem~\ref{thm:IX}, \ref{thm:ILambda}, and \ref{thm:IXirreg}. 
Their derivation led to the following observations.
\begin{itemize}
\item While it is possible to localize interpolation errors in the $H^{-1}(\Omega)$ as for example done in \cite{DieningStornTscherpel22}, it is not possible to localize the $L^2(\mathcal{J};H^{-1}(\Omega))$ error in space without introducing a negative power of the local mesh size as weight; see Remark~\ref{rem:LoclHmin}.
\item The parabolic \Poincare inequality in Theorem~\eqref{thm:paraPoincare} suggest a parabolic scaling $h_t(K) \eqsim h_x^2(K)$ for the interpolation of irregular functions $v \in X$. This scaling occurs also when we change the norm in our interpolation error estimates like in \eqref{eq:changeOfNorms}. Roughly speaking, this change of norms reads
\begin{align*}
\qquad\quad\lVert \partial_t \bigcdot \rVert_{L^2(K_t;H^{-1}(K_x))} \approx \frac{1}{h_t} \lVert \bigcdot \rVert_{L^2(K_t;H^{-1}(K_x))}  \approx \frac{h_x^2}{h_t} \lVert \nabla_x \bigcdot \rVert_{L^2(K_t;L^2(K_x))}.
\end{align*}
On irregular meshes, we have to use the $L^2(\mathcal{J};L^2(\Omega))$ norm to localize the error in the approximation of the time derivative. If we change the norm (which we have to do according to Remark~\ref{rem:paraVsLocal}), we observe roughly speaking
\begin{align*}
\qquad\quad\lVert \partial_t \bigcdot \rVert_{L^2(K_t;L^2(K_x))} \approx \frac{1}{h_t} \lVert \bigcdot \rVert_{L^2(K_t;L^2(K_x))}  \approx \frac{h_x}{h_t} \lVert \nabla_x \bigcdot \rVert_{L^2(K_t;L^2(K_x))}.
\end{align*}
This indicates that parabolic scaling occurs naturally for tensor product meshes but causes difficulties for irregular meshes. Remark~\ref{rem:paraVsLocal} underlines the latter observation.
\end{itemize}
All in all, we have shown that the localization of the $L^2(\mathcal{J};H^{-1}(\Omega))$ norm in space leads to some unavoidable difficulties, which can partially be overcome by assuming additional smoothness of the underlying function. However, interpolation operators $\mathcal{I}\colon X \to X_h$ cannot have the same beneficial properties as interpolation operators for elliptic problems. It is likely that similar difficulties occur in the numerical analysis of simultaneous space-time variational formulations, in particular when the underlying mesh is irregular. For example, to the authors' knowledge there exists no numerical scheme that leads to quasi-optimal approximations with respect to the norm in $X$ with underlying meshes that do not have some kind of tensor product structure; c.f.~\cite{StevensonWesterdiep20}. An exception are minimal residual methods  \cite{FuehrerKarkulik19,GantnerStevenson20,DieningStorn21}, which are quasi-optimal in a slightly stronger norm. This might indicate that the norm in $X$ is actually not well suited for adaptive numerical schemes and a remedy might be the use of alternative norms that are better suited for localization.

\printbibliography

\end{document}